\documentclass[12pt]{article}
\usepackage{amsfonts}
\usepackage{mathrsfs}
\usepackage[T1]{fontenc}
\usepackage[utf8]{inputenc}
\usepackage[letterpaper]{geometry}
\usepackage{amsmath}
\usepackage{graphicx}
\usepackage[active]{srcltx}
\usepackage{amssymb}
\usepackage{color}

\newtheorem{theorem}{Theorem}[section]

\newtheorem{lemma}[theorem]{Lemma}

\newtheorem{proposition}[theorem]{Proposition}
\newtheorem{remark}[theorem]{Remark}

\newenvironment{proof}[1][Proof]{\textbf{#1.} }{\ \rule{0.5em}{0.5em}}

\newtheorem{prop}[theorem]{Proposition}

\def\R{{\Bbb R}}
\def\E{{{\Bbb E}\,}}
\def\cF{{\cal F}}

\def\RR{\mathbb{R}}
\def\E{\mathbb{E}}
\def\cG{{\cal G}}

\def\si{{\sigma}}
\def \eref#1{\hbox{(\ref{#1})}}

\def\ola{\overleftarrow}

\allowdisplaybreaks

\begin{document}

\title{Nonlinear Feynman-Kac formulae for SPDEs with space-time noise}

\author{Jian  Song, Xiaoming Song and Qi Zhang}
 \date{}
\maketitle

\begin{abstract}
 We study a class of backward doubly stochastic differential equations (BDSDEs) involving  martingales with spatial parameters, and show that they provide probabilistic interpretations (Feynman-Kac formulae) for certain semilinear stochastic partial differential equations (SPDEs) with space-time noise. As an application of the Feynman-Kac formulae, random periodic solutions and stationary solutions to certain SPDEs are obtained.
\end{abstract}


\vspace{1mm}


\section{Introduction}

The existence and uniqueness of  solutions to  general backward stochastic differential equations (BSDEs) was obtained by Pardoux and Peng in their pioneering work \cite{pp90}, and they found in \cite{pp} that solutions to BSDEs  provide probabilistic interpretations for solutions to semilinear parabolic PDEs, which is an extension of the classical Feynman-Kac formula. Furthermore, Pardoux and Peng \cite{pardouxpeng} introduced and studied the so-called backward doubly stochastic  differential equations (BDSDEs), the solutions to  which serve as  nonlinear Feynman-Kac formulae for associated semilinear SPDEs driven by white noise in time. Along this line, this article concerns probabilistic interpretations (nonlinear Feynman-Kac formulae) for solutions to a class of semilinear SPDEs driven by space-time noise.

Let $(\Omega, \mathcal F, P)$ be a   probability space satisfying the usual conditions. Let $W=(W_t, t\ge0)$ be standard $d$-dimensional Brownian motion and $(B(t,x), t\ge0 )$ be a  one-dimensional local martingale with spatial parameter $x\in \R^d$ which is  independent of $W$.  Consider the following BDSDE,
 \begin{eqnarray}\label{bdsde}
Y_s^{t,x}&=&\phi(X_T^{t,x})+\int_s^T f(r,X_r^{t,x},Y_r^{t,x},Z_r^{t,x})dr\nonumber\\
 &&\quad\quad+{ \int_s^T g(r,X_r^{t,x},Y_r^{t,x},Z_r^{t,x} )\overleftarrow B(dr,X_r^{t,x})}-\int_s^T Z_r^{t,x}dW_r,
 \quad  s\in[t, T],
\end{eqnarray}
where $X_s^{t,x}$ is the unique strong solution to
 \begin{equation}\label{sde}
dX_s^{t,x}=b(X_s^{t,x})ds+\sigma(X_s^{t,x})dW_s, \quad s\in[t,T],\quad
X_t^{t,x}=x\in \R^d.
\end{equation}
Here $\phi: \R^d\rightarrow \R; f,g: [0,T]\times\R^d\times \R\times \R^{d}; b: \RR^d\rightarrow \RR^d$ and $\si:\RR^d\rightarrow \RR^{d\times d}$ are  measurable functions.  Denote by $\mathscr L$ the infinitesimal generator of $X$, i.e.
\[
(\mathscr L u)(x) =\frac12 \sum_{i,j=1}^d a_{ij}(x)\frac{\partial ^2  }{\partial x_i\partial x_j}u(x)
+\sum_{i =1}^d b_{i }(x)\frac{\partial    }{\partial x_i } u(x)\,,
\]
with $a_{ij}(x)=\sum_{k=1}^d \si_{ik}(x) \si_{jk}(x).$

 There are three major goals in this article. Firstly, under suitable conditions, we obtain an  existence and uniqueness theorem for BDSDE \eqref{bdsde}. Secondly, we  establish the connection between BDSDE \eqref{bdsde} and the following semilinear SPDE,
\begin{equation}\label{spde}
\begin{cases}
-\partial_tu(t,x)= \Big[\mathscr L u(t,x)+f(t,x,u(t,x), \nabla u(t,x)\sigma(x))\Big]dt\\
\quad \qquad\qquad  \qquad +g(t,x,u(t,x),\nabla u(t,x) \sigma(x) )\overleftarrow B(dt, x),\quad\quad  (t,x)\in [0,T]\times \mathbb R^d,\\
u(T,x)=\phi(x)\,,
\end{cases}
\end{equation}
where  $\nabla u=(\partial_{x_1} u, \dots, \partial_{x_d} u)$ and  $\int  \overleftarrow B(dt,x)$ is a backward It\^o integral. Thirdly, as an application of this connection, we construct periodic and stationary solutions for SPDEs via infinite horizon BDSDEs. 

 We would like to give a  brief review for the background and some remarks on the connection between our work and some related literature.  After the introduction of BDSDEs driven by two independent Brownian motions in  \cite{pardouxpeng}, BDSDEs and probabilistic interpretations (nonlinear Feynman-Kac formulae) for SPDEs have been extensively investigated in several directions, and we list  a few of them which is far from complete.

    For SPDEs driven by temporal white noise, Bally and Matoussi  gave  probabilistic interpretations for  solutions in Sobolev spaces (weak solutions) in \cite{bm01}; Buckdahn and Ma  introduced  viscosity solutions and  established Feynman-Kac formulae  in \cite{bma01, bma01'}.  SPDEs driven by temporal colored noise and the associated BDSDEs driven by Brownian motion and fractional Brownian motion were studied  by Jing and Le\'on in \cite{jing, jingleon}.

    Feynman-Kac formulae for linear SPDEs with space-time noise were obtained in  \cite{hhlnt, hns, hns11, s17}.  On the other hand, the related results for nonlinear SPDEs with space-time noise seem to be very limited, and we only find one paper \cite{MS} which is due to  Matoussi and Scheutzow.  In \cite{MS}, the authors dealt with a general type of SPDEs with  nonlinear space-time noise which includes equation \eqref{spde}.  However, the conditions imposed in \cite{MS} for the associated BDSDEs  are rather restrictive in our situation. In the present article,  we obtain the existence and uniqueness of the solution to BSDE \eqref{bdsde} under relatively general conditions (see Theorem \ref{BDSDE} and Remark \ref{link-MS} for its relationship with the result in  \cite{MS}).  In comparison with BDSDEs driven by (fractional) Brownian motions in \cite{pardouxpeng, bm01, bma01', jingleon}, technically it is more difficult to establish an existence and uniqueness theorem under general conditions for BDSDE \eqref{bdsde} due to the spatial dependence of $B(t,x)$. The key step is to combine the It\^o's formula with contraction mapping theorem in a proper way to obtain the existence and uniqueness of the solution  in a suitable Banach space (i.e. $\mathcal M^{2,\beta}$ given in \eqref{mspace}, see the proof of Theorem \ref{BDSDE}).

A remarkable application of the nonlinear Feynman-Kac formula is the construction of random periodic and stationary solutions to SPDEs via the associated BDSDEs.  Unlike the deterministic situation, in which elliptic PDEs give the steady status of parabolic PDEs when time tends to infinity, ``elliptic SPDEs'' do not exist, and we need to use other equations to take the role of ``elliptic SPDEs''. It turns out that the solutions to associated infinite horizon BDSDEs can describe the periodic/stationary solutions to SPDEs.\footnote{Note that Peng \cite{pe} first discovered that the solutions to semilinear elliptic PDEs can be represented by the solutions to infinite horizon BSDEs, and Zhang and Zhao \cite{zhangzhao} obtained random stationary solutions to SPDEs driven by cylindrical Brownian motion via BDSDEs.} In Section 6, we aim to find the random periodic solution to the following infinite horizon SPDE without terminal value which has a form on the interval $[0,T]$ for arbitrary $T>0$:
\begin{eqnarray}\label{zhang685}
u(t,x)&=&u(T,x)+\int_{t}^{T}[\mathscr{L}u(s,x)+f\big(s,x,u(s,x),(\sigma^T\nabla u)(s,x)\big)]ds\nonumber\\
&&-\int_{t}^{T}g\big(s,x,u(s,x),(\sigma^T\nabla v)(s,x)\big)\overleftarrow B(ds, x).
\end{eqnarray}
For a given $\tau>0$, if a solution $u$ to SPDE (\ref{zhang685}) satisfies
\begin{eqnarray}\label{sz6}
{\theta}_{\tau}\circ u({t,\cdot})=u({t+\tau,\cdot})\ \ \ {\rm for}\
{\rm all}\ t\geq0,\ \ {\rm a.s.,}
\end{eqnarray}
where $\theta$ is the shift operator defined in Section 6, we call $u$ a random periodic solution.

Periodicity is a common phenomenon in our world which is exhibited in, for instance,  change of seasons, long-duration oscillation of ocean temperature, and migration pattern of birds. Many efforts have been made by  mathematicians, physicists, oceanographers, biologists, etc., to depict and study  periodicity in systems perturbed by noises. Considering the significance of periodic solution in deterministic dynamical system, the importance of (random) periodic solution in random dynamical system is obvious. However, unlike in  deterministic dynamical systems, the perturbations caused by the noises in  random dynamical systems break the strict periodicity,  which had brought  difficulty to give a rigorous mathematical definition of periodicity for a long time. Observing that the random periodic solution is a stationary solution (stochastic fixed solution) of fixed discrete times with an equal interval as the period, Zhao and Zheng \cite{zheng} put forward the concept of random periodic solution for $C^1$-cocycles, and later  Feng, Zhao and Zhou proposed random periodic solution for semi-flows  in \cite{fe-zh-zh}. 

Nevertheless, random periodic solutions can be obtained in  few cases for SPDEs due to the partial differential operator and the noises. To our best knowledge, the only known result was  obtained by  Feng, Wu and Zhao in \cite{fe-wu-zh} based on the definition of random periodic solution for semi-flows.   In \cite{fe-wu-zh} , the authors identified random periodic solutions to SPDEs driven by temporal white noise  with  solutions to infinite horizon random integral equations.   In this article, random periodic solutions to SPDEs driven by  local martingales with spatial parameters are constructed by the associated BDSDEs. We would like to point out that our result  is not an immediate extension of \cite{fe-wu-zh} or \cite{zhangzhao},  since the  noise in SPDE \eqref{spde}   also depends on the space variable $x$, which makes the analysis more challenging.

This article is organized as follows. In Section 2, we recall some preliminaries on It\^o-Kunita's stochastic integral and provide some lemmas which shall be used later.    The existence and uniqueness of the solution to BDSDE \eqref{bdsde} is studied in Section 3, and the $p$-moments of the solution is estimated in Section 4. In Section 5, with the help of the finite $p$-moments of the solution, we obtain the regularity of the solution to the BDSDE and then establish the connection between BDSDE \eqref{bdsde} and SPDE \eqref{spde}. Finally, in Section 6 we construct the periodic and stationary solution to the SPDE via the infinite horizon BDSDE.

Throughout the paper, $C$ is a generic constant which may vary in different places.

\setcounter{equation}{0}

\section{Some preliminaries}

In this section we provide preliminaries on the integrals against local martingales with spatial parameters and some useful lemmas.  For more details on the It\^o-Kunita's integral, we refer to \cite{kunita}.

Denote
\begin{eqnarray}\label{Ft}
& \mathcal F_t
 =  \mathcal F_t^W \vee \mathcal F_{t,T}^B \,,\quad
 & \mathcal G_t =  \mathcal F_t^W \vee \mathcal F_{T}^B\,,
\end{eqnarray}
where $\mathcal F_{s,t}^B  = \sigma\{B(r,x)-B(s,x), s\le r\le t, x\in\mathbb R^d\}$, $\mathcal F_t^B=\mathcal F_{0,t}^B$,  and $\cF_{s,t}^W$ and $\cF_t^W$ are defined in a similar way.  Note that $\mathcal G_t$ is  a filtration, while $\mathcal F_t$ is not. The joint quadratic variation of $(B(s,x), t\ge 0, x\in\R^d)$ is
\begin{equation}
\langle B(\cdot, x)\,, B(\cdot, y)\rangle_t =\int_0^t q(s, x, y) ds\,.\label{e.1.1}
\end{equation}
Throughout the paper,  we assume the following condition  on $q(s,x,y)$.
\begin{description}
\item{\bf (H)} The function $q(s,x,y)$ given in \eqref{e.1.1} satisfies
\begin{equation}\label{e.1.2}
\sup_{0\le s\le T}|q(s, x,y)|\le K (1+|x|^\kappa+|y|^\kappa),
\end{equation}
 for some $ 0<\kappa<2 $ and $0<K<\infty$.
\end{description}

%

 Let $(f_t,0\le t\le T)$ be a predictable process with respect to the backward filtration $\mathcal F_{t,T}^B$ satisfying
\begin{equation}\label{qcondition}
\int_0^T q(s,f_s,f_s)ds<\infty\qquad a.s.,
\end{equation}
then the stochastic integral $\int_0^T\overleftarrow B(ds, f_s)$
is well-defined (see e.g. \cite[Chapter 3]{kunita}). In particular, if the paths of $f_t$ are a.s. continuous, then assuming {\bf (H)},  condition (\ref{qcondition}) is satisfied, and the integral can be approximated by Riemann sums (\cite{kunita, mckean})
\[ \int_t^T \overleftarrow B(ds, f_s)=\lim_{|\Delta|\to 0} \sum_{k=0}^{n-1} \big[ B(t_{k+1}, f_{t_{k+1}})- B(t_{k}, f_{t_{k+1}}) \big],\]
where $\Delta=\{ t=t_0<\cdots<t_n=T\}$ and   $|\Delta|=\sup_{0\le k \le n-1}|t_{k+1}-t_k|$.

Let $(X_s^{t,x}, t\le s\le T)$ be the solution to equation (\ref{sde}).  Note that $X^{t,x}$ is independent of $B$ and it is a.s. continuous. Thus $\int_s^T \overleftarrow B(dr, X_r^{t, x})$ for $t\le s \le T$ is well defined under  {\bf(H)}, and
 its quadratic variation is given by (\cite[Theorem 3.2.4]{kunita})
\begin{equation}\label{eq-quad}
\left\langle \int_\cdot^T \overleftarrow B(dr, X_r^{t, x}) \right\rangle_{s,T}=\int_s^T q(r,X_r^{t,x}, X_r^{t,x}) dr\,.
\end{equation}

In the sequel, we shall use the following generalized It\^o's formula, which is an extension of   \cite[Lemma 1.3]{pardouxpeng}.
\begin{lemma}\label{itolemma}
 Suppose that $f, g$ and $h$ are $\cF_t$-measurable processes such that 
 \[\int_0^T |f_s|ds+\int_0^T g_s^2 q(s,X_s,X_s)ds+\int_0^T h_s^2 ds<\infty,\, a.s.,\]
 where $X_s=X_s^{0,x}$.  Let $S_t$ be $\cF_t$-measurable  and of the form
$$S_t=S_0+\int_0^t f_s ds+\int_0^t g_sd\ola B(ds, X_s)+\int_0^t h_s dW_s,
\, 0\le t\le T,$$
then we have
\begin{eqnarray}
S_t^2
&=& S_0^2+2\int_0^t S_sdS_s -\int_0^t|g_s|^2q(s,X_s,X_s)ds+\int_0^t|h_s|^2ds\notag\\
&=& S_0^2+2\int_0^t S_sf_s ds+2\int_0^t S_sg_sd\ola B(ds, X_s)
+2\int_0^t S_sh_sdW_s\nonumber\\
&&\quad  -\int_0^t|g_s|^2q(s,X_s,X_s)ds+\int_0^t|h_s|^2ds\,.\label{e.3.5}
\end{eqnarray}
 More generally, for any function $\varphi\in C^2(\R)$, we have the following It\^o's formula,
 \begin{eqnarray}
\varphi(S_t)
&=& \varphi(S_0)+\int_0^t \varphi'(S_s)dS_s -\frac12\int_0^t\varphi''(S_s)
|g_s|^2q(s,X_s,X_s)ds+\frac12\int_0^t\varphi''(S_s)|h_s|^2ds\notag\\
&=& \varphi(S_0)+\int_0^t \varphi'(S_s)f_s ds+\int_0^t \varphi'(S_s)g_sd\ola B(ds, X_s)
+\int_0^t \varphi'(S_s)h_sdW_s\nonumber\\
&&\quad -\frac12\int_0^t\varphi''(S_s)
|g_s|^2q(s,X_s,X_s)ds+\frac12\int_0^t\varphi''(S_s)|h_s|^2ds\,.\label{e.ito}
\end{eqnarray}

\end{lemma}
\begin{proof}
The It\^o's formula \eref{e.ito} can be proven by the standard methods of approximation and localization (see e.g. \cite[Theorem 3.3]{KaratzasShreve}). Here we provide a sketch of proof for \eref{e.3.5}, and \eref{e.ito} can be proven in a similar spirit.

Fixing $t>0$ and a partition $\Delta=\{0=t_0, t_1,\ldots, t_n=t\}$ of $[0,t]$, we have
\begin{align}
 S_t^2-S_0^2=&\sum_{i=0}^{n-1} \left[S_{t_{i+1}}^2-S_{t_i}^2\right]\notag\\
 =&2\sum_{i=0}^{n-1} S_{t_{i}}\left[S_{t_{i+1}}-S_{t_i}\right]+\sum_{i=0}^{n-1} \left[S_{t_{i+1}}-S_{t_i}\right]^2\notag\\
 =&2\sum_{i=0}^{n-1}S_{t_i}\left(\int_{t_i}^{t_{i+1}} f_s ds+\int_{t_i}^{t_{i+1}} h_s dW_s\right)+2\sum_{i=0}^{n-1}S_{t_{i+1}} \int_{t_i}^{t_{i+1}} g_s d\ola B(ds, X_s)\notag\\
 & -2\sum_{i=0}^{n-1}(S_{t_{i+1}}-S_{t_i}) \int_{t_i}^{t_{i+1}} g_s d\ola B(ds, X_s)+\sum_{i=0}^{n-1} \left[S_{t_{i+1}}-S_{t_i}\right]^2.\label{eq2.5}
\end{align}
When the mesh size $|\Delta|$ goes to zero, $$\sum_{i=0}^{n-1}S_{t_i}\left(\int_{t_i}^{t_{i+1}} f_s ds+\int_{t_i}^{t_{i+1}} h_s dW_s\right)+\sum_{i=0}^{n-1}S_{t_{i+1}} \int_{t_i}^{t_{i+1}} g_s d\ola B(ds, X_s)$$ converges to $\int_0^t S_s dS_s$, and the rest terms on the right-hand side of \eqref{eq2.5}
\begin{align*}
&-2\sum_{i=0}^{n-1}(S_{t_{i+1}}-S_{t_i}) \int_{t_i}^{t_{i+1}} g_s d\ola B(ds, X_s)+\sum_{i=0}^{n-1} \left[S_{t_{i+1}}-S_{t_i}\right]^2\\
&=\sum_{i=0}^{n-1} \left(S_{t_{i+1}}-S_{t_i}\right)\left(\int_{t_{i}}^{t_{i+1}} h_s dW_s-\int_{t_{i}}^{t_{i+1}} g_s d\ola B(ds, X_s) \right) +\sum_{i=0}^{n-1} \rho_i\\
&=\sum_{i=0}^{n-1}\left(\int_{t_{i}}^{t_{i+1}} h_s dW_s+\int_{t_{i}}^{t_{i+1}} g_s d\ola B(ds, X_s) \right)\left(\int_{t_{i}}^{t_{i+1}} h_s dW_s-\int_{t_{i}}^{t_{i+1}} g_s d\ola B(ds, X_s) \right)+\sum_{i=0}^{n-1} \rho_i'\\
&=\sum_{i=0}^{n-1}\left(\int_{t_{i}}^{t_{i+1}} h_s dW_s\right)^2 -  \sum_{i=0}^{n-1}\left( \int_{t_{i}}^{t_{i+1}} g_s d\ola B(ds, X_s) \right)^2+\sum_{i=0}^{n-1} \rho_i'
\end{align*}
converges to $$\int_0^t h_s^2ds- \int_0^t g_s^2 q(s,X_s,X_s)ds,$$ since $\sum_{i=0}^{n-1} \rho_i'$ converges to zero based on the fact that the covariation between a martingale and an absolutely continuous process is zero. \hfill
\end{proof}

Similarly, we also have the following product rule.
\begin{lemma}\label{product}
 Let $Q_t$ be a continuous $\cF_t$-measurable process with bounded variation and $S_t$ be given in Lemma \ref{itolemma}. Then the following product rule holds
 \[d(S_tQ_t)= S_t dQ_t +Q_t dS_t.\]
\end{lemma}
%
The following result of exponential integrability will be used in the proof the existence and uniqueness of the solutions to \eqref{bdsde} in Section 3.

\begin{lemma}\label{exponential} If $b$ and $\sigma$ are bounded measurable functions, we have
\[\E \int_0^T \exp\left(p q(t,X_t,X_t)\right) dt<\infty, \text{ for all } p >0.\]
In particular,
\[\E \exp\left(p \int_0^T q(t,X_t,X_t)dt\right)<\infty \text{ and } ~ \E \int_0^T q(t,X_t,X_t)^p dt<\infty, \text{ for all } p >0.\]

\end{lemma}

\begin{proof}
Note that  $X_t=x+\int_0^t b(X_s)ds+\int_0^t \sigma(X_s)dW_s.$
 Since $b$ is bounded and $|q(t,x,y)|\le K(1+|x|^\kappa+|y|^\kappa)$ with  $\kappa\in(0, 2)$  for all $t\in[0,T]$ by condition {\bf (H)}, it suffices to show that
\[\E \int_0^T \exp\left(p \left|\int_0^t\sigma(X_s) dW_s\right|^\kappa \right) dt<\infty, \text{ for all } p >0,\]
which can be reduced to show that
\begin{equation}\label{e2.7}
\E\left[\sup_{0\le t\le T} \exp\left(p \left|\int_0^t \sigma(X_s) dW_s\right|^\kappa \right) dt \right]<\infty.
\end{equation}
Denoting $N_t=\int_0^t \sigma(X_s) dW_s$, by the exponential inequality for martingales (see, e.g. \cite[Formula (A.5)]{nualart}),
we have for any $x>0$,
 \begin{equation}\label{e2.8}
 P\left(\sup_{0\le t\le T}|N_t|\ge x\right)\le 2\exp\left(-\frac{x^2}{2D_0}\right),
 \end{equation}
 where $D_0=T\|\sigma\|_\infty^2<\infty.$  Let $\tilde N=\sup_{0\le t\le T]}|N_t|$. The left-hand side of (\ref{e2.7}) is estimated as follows,
 \begin{align*}
&\E\left[\sup_{0\le t\le T} \exp\left(p \left|\int_0^t \sigma(X_s) dW_s\right|^\kappa \right) dt \right]=\E\left[e^{p \left|\tilde N\right|^\kappa}   \right]\\
&=\int_\R P\left(|\tilde N|^\kappa \ge \frac yp\right) e^y dy \le 1+\int_0^\infty 2\exp\left(-\frac{y^{\frac2\kappa}p^{-\frac2\kappa}}{2D_0} + y\right)dy,
 \end{align*}
 where the integral on the right-hand side is finite for all $p>0$ since $\kappa<2$.  The proof is concluded. \hfill
\end{proof}
\setcounter{equation}{0}
\section{Existence and uniqueness of  solutions to BDSDEs}

This section concerns the existence and uniqueness  theorem for the following BDSDE
\begin{equation}\label{bdsde1}
Y_t=\xi+\int_t^T f(s,Y_s,Z_s)ds+\int_t^T g(s, Y_s,Z_s)\ola B(ds, X_s)
-\int_t^T Z_sdW_s,\, t\in[0,T],
\end{equation}
where $(X_s=X_s^{0,x}, 0\le s\le T)$ is the unique solution to \eref{sde}.   Denote, for $p\ge 1$,
\begin{align*}
&S^p([0,T];\mathbb R)
=\Big\{h:\Omega\times [0,T]\to \RR; \text{ continuous}, h(t) \text{ is $\mathcal F_t$-measurable, and }
   \E\big[\sup_{0\le t\le T} |h(t)|^p\big]<\infty\Big\}\,,\\
&M^p([0,T];\mathbb R^l)
=\Big\{\varphi:\Omega\times [0,T]\to \RR^l;~ \varphi(t) \text{ is $\mathcal F_t$-measurable and }\E\int_0^T |\varphi(t)|^pdt<\infty\Big\}\,,\\
&Q^{2p}([0,T];\mathbb R)
=\Big\{g:\Omega\times [0,T]\to \RR;~ g(t) \text{ is $\mathcal F_t$-measurable and }\E\int_0^T |g(t)|^{2p} |q(t,X_t,X_t)|^p dt<\infty\Big\}\,,
      \end{align*}
where $\cF_t$ is given in \eqref{Ft}.

We will follow the standard procedure in \cite{pardouxpeng}.
First, as a preparation, we prove the following existence and uniqueness result when $f$ and $g$ are independent of $Y$ and $Z$.

\begin{proposition}\label{prop2.2}
Let $f\in M^{2}([0,T]; \mathbb R), g\in Q^{2}([0,T]; \mathbb R)$ and  $\xi\in L^{2}(\mathcal F_T)$.
Then the equation
\begin{equation}
Y_t=\xi+\int_t^T f(s)ds+\int_t^T g(s)d\ola B(ds, X_s)-\int_t^T Z_s
dW_s,~ t\in[0,T],\label{e.linear}
\end{equation}
has a unique solution $(Y,Z)\in S^{2}([0,T];\mathbb R)
\times  M^{2}([0,T];\mathbb R^{  d})$.
\end{proposition}

\begin{proof}
First we discuss the uniqueness.
Suppose $(Y^i,Z^i),i=1,2$, are two solutions in $M^{2}([0,T];\mathbb R) \times M^{2}([0,T];\mathbb R^{  d})$.  Denote $\bar Y=Y^1-Y^2$ and $\bar Z=Z^1-Z^2$.
Then \[\bar Y_t+\int_t^T \bar Z_s dW_s=0,\]
and hence \[\E (\bar Y_t^2)+\E\int_t^T |\bar Z_s|^2ds=0\]
because $\bar Y_t\in \cF_t=\cF_{t}^W\vee \cF_{t,T}^B$ and $\E^W(\bar Y_t\int_t^T \bar Z_sdW_s)=0$, where
$\E^W$ means the expectation taken in the probability space generated by $W$.  This  immediately implies the uniqueness.

Now, we consider the existence.
Denote $d\ola B(ds, X_s)$ by $dM_s$ and let
\[
N_t=\E\left[\left.\xi+\int_0^T f(s)ds+\int_0^T g(s)dM_s\right|\cG_t
\right]\,,
\]
which is a martingale with respect to the filtration $\cG_t=\cF_t^W\vee\cF_T^B$.
Then by the   martingale representation theorem,  there exists a square integrable process $Z_t\in \mathcal G_t$ such  that
\begin{equation}
N_t=N_0+\int_0^t Z_sdW_s.\label{e.z}
\end{equation}
Letting $t=T$ we have  $N_T=N_0+\int_0^TZ_sdW_s$.
On the other hand, from the definition of $N_t$, we have
\[
N_T= \xi+\int_0^T f(s)ds+\int_0^T g(s)dM_s
 \,.
\]
Thus, we have
\begin{eqnarray*}
N_t
&=&N_0+\int_0^t Z_sdW_s=N_T-\int_t^T Z_sdW_s\\
&=&\xi+\int_0^T f(s)ds+\int_0^T g(s)dM_s-\int_t^T Z_sdW_s\,.
\end{eqnarray*}
Namely, we have
\begin{equation}\label{e3.4'}
\E\left[\left.\xi+\int_0^T f(s)ds+\int_0^T g(s)dM_s\right|\cG_t\right]=\xi+\int_0^Tf(s)ds+\int_0^Tg(s)dM_s-\int_t^T Z_sdW_s.
\end{equation}
Let
\begin{equation}
Y_t=\E\left[\left.\xi+\int_t^T f(s)ds+\int_t^T g(s)dM_s\right|\cG_t\right]\,,\label{e.y}
\end{equation}
then by \eqref{e3.4'},
\[Y_t=\xi+\int_t^Tf(s)ds+\int_t^T g(s)dM_s-\int_t^T Z_sdW_s.\]
Thus $(Y_t, Z_t)$ given by \eref{e.z} and \eref{e.y} satisfies \eref{e.linear}.

Now we show $(Y_t,Z_t)\in \mathcal F_t=\mathcal F_{t}^W\vee \mathcal F_{t,T}^B.$ Note that $Y_t\in \mathcal F_t$ because $\xi+\int_t^T f(s)ds+\int_t^T g(s)dM_s\in \mathcal F_{T}^W\vee\mathcal F_{t,T}^B$ and
its  expectation conditional on $\mathcal G_t=\mathcal F_{t}^W\vee\mathcal F_T^B$ is
measurable with respect to $\mathcal F_{t}^W\vee \mathcal F_{t,T}^B.$ \ Similarly, we have
$\int_t^T Z_sdW_s=\xi+\int_t^T f(s)ds+\int_t^T g(s)dM_s-Y_t$ is  $\mathcal F_{T}^W\vee\mathcal F_{t,T}^B$ measurable since the right-hand side is. By the martingale representation theory, $Z_s$ is $\cF_s^W\vee\cF_{t,T}^B$ measurable for $s\in[t,T]$, and hence  $Z_t   $ is $\cF_t$-measurable.

Finally, we show the square integrability of $Y$ and $Z$. By equations \eqref{eq-quad} and (\ref{e.y}), H\"older inequality and Burkholder-Davis-Gundy inequality, there exists a constant $C$  depending only on $T$ such that
\begin{align}
 \E \left(\sup_{0\le t\le T}|Y_t|^{2}\right)\le & C\left(\E(|\xi|^{2})+\E\int_0^T |f(s)|^{2}ds+\E\int_0^Tg^2(s)q(s,X_s,X_s)ds\right)\notag\\
 < & \infty. \label{bound-Y}
\end{align}
Hence $Y\in S^{2}([0,T];\mathbb R)$.
On the other hand, noting
\[\int_0^T Z_s
dW_s=-Y_0+\xi+\int_0^T f(s)ds+\int_0^T g(s)dM_s,\]
by Burkholder-Davis-Gundy inequality, we have
 $Z\in \mathbb M^{2}([0,T];\mathbb R^{  d})$. \hfill
\end{proof}

Now we are ready to prove the main result Theorem \ref{BDSDE} in this section. Assume the following conditions.
\begin{description}
\item{\bf (A1)}\
Let the functions $b$ and $\si$ be bounded and
satisfy the global Lipschitz condition:
 \begin{equation}
 |b(x)-b(y)|+|\si(x)-\si(y)|\le K |x-y|  \,,\label{e.1.3}
 \end{equation}
where we use $|\cdot|$ to denote both the Euclidean norm for a vector in $\RR^d$ and
the Hilbert-Schmidt norm for a matrix in $\RR^{d\times d}$.
\end{description}
\begin{description}
\item{\bf (B1)}\ Let $f$ and $g$ be two given functions such that
$f(\cdot, 0,0)\in M^2([0,T];\R), \, g(\cdot, 0,0)\in Q^2([0,T];\R),$
and
\begin{eqnarray*}
|f(t,y_1,z_1)-f(t,y_2,z_2)|^2
&\le&  K (|y_1-y_2|^2+|z_1-z_2|^2)\,;\\
 |g(t,y_1,z_1 )-g(t,y_2,z_2 )|^2
&\le&  K|y_1-y_2|^2+\alpha_t(\omega)|z_1-z_2|^2   \,,
\end{eqnarray*}
where $\alpha_t(\omega) q(t,X_t,X_t)\le \alpha $ a.s. for some constant $\alpha\in (0,1).$
\end{description}

\begin{theorem}\label{BDSDE} Let  the conditions {\bf (H)}, {\bf (A1)} and {\bf (B1)} be satisfied.
Assume $\xi\in L^{p} (\mathcal F_{T})$ for some $p>2$, the BDSDE \eref{bdsde1}
has a unique solution $(Y,Z)\in \mathcal M^{2,\beta}$ for some $\beta>0$, where the space $\mathcal M^{2,\beta}$ is defined in (\ref{mspace}). Furthermore, $Y\in S^2([0,T];\R)$.
\end{theorem}

\begin{remark}\label{link-MS}
In \cite{MS}, the authors considered the following BDSDE
\begin{equation}\label{bdsde-ms}
Y_s=\xi+\int_s^T f(r, Y_r, Z_r)dr+\int_s^T M(dr, \tilde g(r,Y_r,Z_r))-\int_s^T Z_r dW_r.
\end{equation}
Consider their case when $k=1$ and $l=d+1$.  If we let $M(t,x,y)=yB(t,x)$ for $x\in \mathbb R^d, y\in \mathbb R$ and $\tilde g(r,Y_r,Z_r)=(X_r, g(r, Y_r, Z_r))$, then equation (\ref{bdsde-ms}) is reduced to BDSDE (\ref{bdsde1}). Now the joint quadratic variation of $M$ is given by
\[\langle M(\cdot, x,y), M(\cdot, x',y')\rangle_t =\langle y B(\cdot, x), y'B(\cdot, x')\rangle_t =yy'\int_0^t q(s,x,x' )ds,\]
and thus the characteristic of the family of the local martingales $\{M(\cdot , x,y), (x,y)\in\mathbb R^{d+1}\}$ is $a(s,(x,y),(x',y'))=yy'q(s,x,x')$. In \cite{MS}, to obtain the existence and uniqueness of the solution to BDSDE \eqref{bdsde-ms}, the authors imposed the following condition (inequality (3) on page 5) on the characteristic $a$ when $k=1$
\begin{equation}\label{condition-ms}
|a(s,z,z)-a(s,z,z')-a(s,z',z)+a(s,z',z')|\le |z-z'|^2,
\end{equation}
where $z=(x,y)$ and $z'=(x',y')$.
This condition implies that $\sup\limits_{0\le s\le T,\, x\in\R^d}q(s,x,x)\le 1$ if we let $x=x'$, and that
$\sup\limits_{0\le s\le T, \,y\in \R} y^2|q(s, x,x)-2q(s,x,x')+q(s,x',x')|\le |x-x'|^2$ if we let $y=y',$ which further implies that $q(s,x,y)\equiv C(s)$ for some deterministic function $C(s)\in[0,1].$
\end{remark}

\begin{proof}
Define
\begin{align}\label{mspace}
 \mathcal M^{2,\beta}=\bigg\{(y,z): (\cF_t) \text{-adapted and } &\E\int_0^T \tilde q_s
\exp\left(\beta \int_s^T \tilde q_r dr\right)|y_s|^2ds\notag\\
&\qquad + \E\int_0^T
\exp\left(\beta \int_s^T \tilde q_r dr\right)|z_s|^2ds<\infty\bigg\}
\end{align}
with  $\beta>0$ to be determined later and $\tilde q_r=q(r,X_r,X_r)\vee 1$.

For any $(y,z)\in \mathcal M^{2,\beta}\subset M^{2}([0,T];\mathbb R)\times  M^{2}([0,T];
\mathbb R^{d})$, by condition {\bf(B1)} and Proposition \ref{prop2.2},
there exists a unique pair $(Y,Z)\in S^{2}([0,T];\mathbb R)
\times  M^{2}([0,T];\mathbb R^{d})$ so that
\begin{equation}\label{mapping}
Y_t=\xi+\int_t^T f(s,y_s,z_s)ds
+\int_t^T g(s,y_s,z_s)dM_s-\int_t^T Z_sdW_s,
\end{equation}
where we take the notation  $dM_s=d\ola B(s,X_s)$.

{\bf Step 1.} In this step, we shall show the mapping defined by \eref{mapping} maps  $\mathcal M^{2,\beta}$ to itself, i.e. $(Y,Z)\in \mathcal M^{2,\beta}$, for all $\beta>0.$

Denote $f_s:=f(s,y_s,z_s)$ and $g_s:=g(s,y_s,z_s)$.  Applying  Lemma \ref{itolemma} and Lemma \ref{product} to $ Y_t^2 \exp\left(\beta \int_0^t \tilde q_s^{(n)} ds\right)$ where $\tilde q_s^{(n)}=\tilde q_s\wedge n$, we have
\begin{align}\label{ito1}
& Y_t^2\exp\left(\beta \int_0^t \tilde q_s^{(n)} ds\right)\notag\\
=&\xi^2\exp\left(\beta \int_0^T \tilde q_s^{(n)} ds\right) -\int_t^T \beta \tilde q_s^{(n)} \exp\left(\beta \int_0^s \tilde q_r^{(n)} dr\right)Y_s^2ds
-\int_t^T \exp\left(\beta \int_0^s \tilde q_r^{(n)} dr\right)|Z_s|^2ds\notag\\
&+2\int_t^T \exp\left(\beta \int_0^s \tilde q_r^{(n)} dr\right)Y_s f_sds
+\int_t^T  \exp\left(\beta \int_0^s \tilde q_r^{(n)} dr\right) g_s^2 q(s,X_s,X_s)ds\notag\\
&- 2\int_t^T \exp\left(\beta\int_0^s \tilde q_r^{(n)} dr\right) Y_s Z_s dW_s+2\int_t^T \exp\left(\beta\int_0^s \tilde q_r^{(n)} dr\right) Y_s g_s dM_s.
\end{align}

The expectations of these two stochastic integrals on the right-hand side of the above equation are equal to zero. Here we only show that
\begin{equation}
\E\int_t^T \exp\left(\beta\int_0^s \tilde q_r^{(n)} dr\right) Y_s g_s dM_s=0,\label{integrable}
\end{equation}
and the other one can be proven in a similar way.
In fact, by Burkholder-Davis-Gundy inequality,
\begin{align*}
 &\E\sup_{0\le t\le T}\int_t^T \exp\left(\beta\int_0^s \tilde q_r^{(n)} dr\right) Y_s g_s dM_s\\
 \le& C\E\left(\int_0^T  Y_t^2 g_t^2 q(t,X_t,X_t)dt\right)^{\frac12}\\
 \le & C \E\left( \sup_{0\le t\le T} |Y_t|^2 \int_0^T (g^2(t,0,0)+y_t^2+\alpha_t(\omega)|z_t|^2)q(t,X_t,X_t) dt\right)^\frac12\\
 \le & C \left(\E(\sup_{0\le t\le T}|Y_t|^2)+\E \int_0^T (g^2(t,0,0)+y_t^2+\alpha_t(\omega) |z_t|^2)q(t,X_t,X_t) dt \right)\\
 <&\infty,
\end{align*}
where the last inequality follows from the facts $Y\in S^2([0,T];\R),\ (y,z)\in \mathcal M^{2,\beta}$, condition {\bf (B1)}, Lemma \ref{exponential} and H\"older inequality. This implies \eref{integrable}.

Now taking expectation in the equation \eref{ito1}, we have
\begin{align*}
& \E \left(Y_t^2\exp\left(\beta \int_0^t \tilde q_s^{(n)} ds\right)\right)+\E\int_t^T \beta \tilde q_s^{(n)} \exp\left(\beta \int_0^s \tilde q_r^{(n)} dr\right)Y_s^2ds
+\E\int_t^T \exp\left(\beta \int_0^s \tilde q_r^{(n)} dr\right)|Z_s|^2ds\\
=& \E\left(\xi^2\exp\left(\beta \int_0^T \tilde q_s^{(n)} ds\right)\right)+2\E\int_t^T \exp\left(\beta \int_0^s \tilde q_r^{(n)} dr\right)Y_s f_sds\\
&\qquad \qquad\qquad \qquad\qquad  \qquad\qquad \qquad+\E\int_t^T  \exp\left(\beta \int_0^s \tilde q_r^{(n)} dr\right) g_s^2 q(s,X_s,X_s)ds\\
\le & \E\left(\xi^2\exp\left(\beta \int_0^T \tilde q_s^{(n)} ds\right)\right)+2\sqrt K \E\int_t^T \exp\left(\beta \int_0^s \tilde q_r^{(n)} dr\right)|Y_s| \left(|f(s,0,0)|+|y_s|+|z_s|\right)ds\\
&+2\E\int_t^T  \exp\left(\beta \int_0^s \tilde q_r^{(n)} dr\right)\left(g^2(s,0,0)+K(y_s^2+\alpha_s(\omega)|z_s|^2)\right) q(s,X_s,X_s)ds\\
\le & \E\left(\xi^2\exp\left(\beta \int_0^T \tilde q_s^{(n)} ds\right)\right)+\sqrt K\E \int_t^T \tilde q_s^{(n)}\exp\left(\beta \int_0^s \tilde q_r^{(n)} dr\right)\left(\delta Y_s^2 +\frac{3}{\delta}(|f(s,0,0)|^2+|y_s|^2+|z_s|^2)\right)ds\\
&+K\E\int_t^T  \exp\left(\beta \int_0^s \tilde q_r^{(n)} dr\right)(g^2(s,0,0)+y_s^2+\alpha_s(\omega)|z_s|^2) q(s,X_s,X_s)ds,
\end{align*}
where in the last step we used the fact $2ab\le \delta a^2+\frac{1}{\delta} b^2$ for any $\delta>0.$ Choose $\delta>0$ sufficiently small such that $\beta>\sqrt K\delta,$ then we have
\begin{align}
& (\beta-\sqrt K\delta)\E\int_0^T  \tilde q_s^{(n)} \exp\left(\beta \int_0^s \tilde q_r^{(n)} dr\right)Y_s^2ds
+\E\int_0^T \exp\left(\beta \int_0^s \tilde q_r^{(n)} dr\right)|Z_s|^2ds\notag\\
\le & \E\left(\xi^2\exp\left(\beta \int_0^T \tilde q_s^{(n)} ds\right)\right)+\frac{3\sqrt K}{\delta} \int_0^T \exp\left(\beta \int_s^T \tilde q_r^{(n)} dr\right)(|f(s,0,0)|^2+|y_s|^2+|z_s|^2)ds\notag\\
&+2\E\int_0^T  \exp\left(\beta \int_0^s \tilde q_r^{(n)} dr\right)\left(g^2(s,0,0)+K(y_s^2+\alpha_s(\omega)|z_s|^2)\right) q(s,X_s,X_s)ds\notag\\
\le & \E\left(\xi^2\exp\left(\beta \int_0^T \tilde q_s ds\right)\right)+\frac{3\sqrt K}{\delta}\E \int_0^T \exp\left(\beta \int_0^s \tilde q_r dr\right)(|f(s,0,0)|^2+|y_s|^2+|z_s|^2)ds\notag\\
&+2\E\int_0^T  \exp\left(\beta \int_0^s \tilde q_rdr\right)\left(g^2(s,0,0)+K(y_s^2+\alpha_s(\omega)|z_s|^2)\right) q(s,X_s,X_s)ds. \label{e3.13}
\end{align}
Here $\E\left(\xi^2\exp\left(\beta \int_0^T \tilde q_s ds\right)\right)<\infty$ because of  Lemma \ref{exponential} and the condition $\xi \in L^p$ for some $p>2$.  The last two integrals in (\ref{e3.13}) are both finite because of Lemma \ref{exponential}, conditions {\bf (B1)} and the fact $(y, z)\in \mathcal M^{2,\beta}.$

Now let $n$ go to infinity, then denoting by $C(\delta)$ the sum of the terms on the right-hand side of (\ref{e3.13}), by the Monotone Convergence Theorem we have
\[(\beta-\sqrt K\delta)\E\int_0^T  \tilde q_s \exp\left(\beta \int_t^T \tilde q_r dr\right)Y_s^2ds
+\E\int_0^T \exp\left(\beta \int_t^T \tilde q_r dr\right)|Z_s|^2ds\le C(\delta)<\infty.\]
Therefore $(Y,Z)\in \mathcal M^{2,\beta}$, and the mechanism (\ref{mapping}) defines a mapping $\Psi$ from $ \mathcal M^{2,\beta}$ to itself:
 $(Y,Z)=\Psi(y,z)$.

{\bf Step 2.} In this step, we shall prove that $\Psi$ is a contraction mapping on $\mathcal M^{2,\beta}$ for sufficiently large $\beta$.

Let $(Y^i,Z^i)=\Psi(y^i,z^i)$ for $i=1,2$,  $\bar f_s= f(s,y^1_s,z^1_s)- f(s,y^2_s,z^2_s)$, 
and $\bar g_s= g(s,y^1_s,z^1_s)-g(s,y^2_s,z^2_s)$. Then
\[ Y^1_t-Y^2_t=\int_t^T \bar f_sds+\int_t^T \bar g_s dM_s-\int_t^T \bar Z_s dW_s.\]
We shall use a generic notation $\bar h=h^1-h^2$, where $h$ can be $Y, Z, y$ and $z$. Applying  Lemma \ref{itolemma} and Lemma \ref{product} to $\bar Y_t^2 \exp\left(\beta \int_0^t \tilde q_s^{(n)} ds\right)$, we have
\begin{align*}
&\bar Y_t^2\exp\left(\beta \int_0^t \tilde q_s^{(n)} ds\right)\\
=& -\int_t^T \beta \tilde q_s^{(n)} \exp\left(\beta \int_0^s \tilde q_r^{(n)} dr\right)\bar Y_s^2ds
-\int_t^T \exp\left(\beta \int_0^s \tilde q_r^{(n)} dr\right)|\bar Z_s|^2ds\\
&+2\int_t^T \exp\left(\beta \int_0^s \tilde q_r^{(n)} dr\right)\bar Y_s\bar f_sds
+\int_t^T  \exp\left(\beta \int_0^s \tilde q_r^{(n)} dr\right) \bar g_s^2 q(s,X_s,X_s)ds\\
&- 2\int_t^T \exp\left(\beta\int_0^s \tilde q_r^{(n)} dr\right) \bar Y_s \bar Z_s dW_s+2\int_t^T \exp\left(\beta\int_0^s \tilde q_r^{(n)} dr\right) \bar Y_s \bar g_s dM_s.
\end{align*}
As in {\bf Step 1}, we can show that these two stochastic integrals on the right-hand side of the above equation are integrable and hence the expectations of them are zero. Taking expectation and letting $n$ go to infinity, we have
\begin{align*}
&\E\left(\bar Y_t^2\exp\left(\beta \int_0^t \tilde q_s ds\right)\right)
+\E\int_t^T \beta \tilde q_s \exp\left(\beta \int_0^s \tilde q_r dr\right)\bar Y_s^2ds
+\E\int_t^T \exp\left(\beta \int_0^s \tilde q_r dr\right)|\bar Z_s|^2ds\\
=&2\E\left(\int_t^T \exp\left(\beta \int_0^s \tilde q_r dr\right)\bar Y_s\bar f_sds
+\int_t^T  \exp\left(\beta \int_0^s \tilde q_r dr\right) \bar g_s^2 q(s,X_s,X_s)ds\right)\\
\le & \E \left[\int_t^T \exp\left(\beta \int_0^s \tilde q_r dr\right)\left(2 \sqrt K|\bar Y_s|
(|\bar y_s|+|\bar z_s|)+K|\bar y_s|^{2} +\alpha_s(\omega)|\bar z_s|^2 q(s,X_s,X_s)\right) ds\right]\\
\le & \E\left[\int_t^T \exp\left(\beta \int_0^s \tilde q_r dr\right) \left(\frac 2a |\bar Y_s|^2 +aK|\bar y_s|^2+aK|\bar z_s|^2+ K|\bar y_s|^{2} q(s,X_s,X_s)+\alpha  |\bar z_s|^2\right) ds\right]\\
\le& \frac {2}a ~ \E\int_t^T \tilde q_s
\exp\left(\beta \int_0^s \tilde q_r dr\right) |\bar Y_s|^2ds
+ (Ka+K)  \E\int_t^T \tilde q_s \exp\left(\beta \int_0^s \tilde q_r dr\right)|\bar y_s|^2ds \\
& \quad \quad +(Ka+\alpha) \E\int_t^T \exp\left(\beta \int_0^s \tilde q_r dr\right) |\bar z_s|^2ds,
\end{align*}
where in the second inequality we used the fact $2xy\le \frac 1a x^2+a y^2$ for any $a>0.$
This implies
\begin{align*}
& \left(\beta-\frac{2}{a}\right)\E\int_0^t  \tilde q_s \exp\left(\beta \int_s^T \tilde q_r dr\right) \bar Y_s^2ds
+\E\int_0^t \exp\left(\beta \int_s^T \tilde q_r dr\right)|\bar Z_s|^2ds \\
 \le &   K(a+1) \E\int_0^t \tilde q_s \exp\left(\beta \int_s^T \tilde q_r dr\right)|\bar y_s|^2ds+(Ka+\alpha)  \E\int_0^t \exp\left(\beta \int_s^T \tilde q_r dr\right) |\bar z_s|^2ds.
\end{align*}
Since $0<\alpha<1$, we may choose $a>0$ so that $Ka+\alpha<1.$ Choose $\beta$ such that $\beta-\frac {2}a=\frac{K(a+1)}{Ka+\alpha}.$ Let $\rho=\frac{K(a+1)}{\beta-\frac{2}a}= Ka+\alpha<1,$
then we have
\[
\|(\bar Y,\bar Z)\|_{2,\beta}^2\le \rho \|(\bar y,\bar z)\|_{2,\beta}^2,
\]
where \[\|(y, z)\|_{2,\beta}^2=\E\int_0^T \tilde q_s
\exp\left(\beta \int_s^T \tilde q_r dr\right)|y_s|^2ds+ \frac{Ka+\alpha}{K(a+1)}\E\int_0^T
\exp\left(\beta \int_s^T \tilde q_r dr\right)|z_s|^2ds.\]
Therefore $\Psi$ is a contraction mapping in the space $\mathcal M^{2,\beta}$ endorsed with the norm $\|\cdot\|_{2,\beta}$.

{\bf Step 3.} Finally, the fact that $Y\in S^2([0,T];\R)$ can be proven in the same way as in Proposition \ref{prop2.2}.\hfill
\end{proof}

\setcounter{equation}{0}


\setcounter{equation}{0}

\section{Moments of the solution}
In this section, we show that under suitable conditions, the solution to \eqref{bdsde1} has finite $p$-moments for $p>2$, which will be used in Section 5 to obtain the regularity of the solution.

Basic calculations yield the following lemma on $\phi_n(x)$, where $\phi_n(x)$ is an approximation of $|x|^{2p}$  at quadratic growth as $|x|$ tends to infinity. The lemma will be used in the proof of Theorem \ref{t.5.2}.
\begin{lemma}\label{l.5.1}
Fix $p>1$. For $x\in \R$, define
 $$\varphi_n(x)=(|x|\wedge n)^p+pn^{p-1}(|x|-n)^{+},$$
 then $\varphi_n(x)$ is a convex function with
 \[\varphi_n'(x)=p|x|^{p-1}I_{[|x|\le n]}+pn^{p-1}I_{[|x|>n]},\quad \varphi_n''(x)=p(p-1)|x|^{p-2}I_{[|x|\le n]},\]
 where $\varphi_n''(x)$ is defined as the Randon-Nikodym derivative $d\varphi_n'(x)/dx$. 
 
 Let $ \phi_n(x)=\varphi_n(x^2).$  Then we have
 \begin{equation*}
 |\phi_n'(x)\cdot x|\le 2p\phi_n(x),\quad |\phi_n''(x)\cdot x^2|\le 2p(2p-1)\phi_n(x).
 \end{equation*}
 Furthermore, we also have the estimations
 \begin{align*}
&|\phi_n'(x)|\le (2p)^{\frac12}\left(\phi_n(x)\right)^\frac12 \left(\phi_n''(x)\right)^\frac12,\\ 
&|\phi_n'(x)|^{\frac{2p}{2p-1}} \le (2p)^{\frac{2p}{2p-1}} \phi_n(x),\\ 
&|\phi_n''(x)|^{\frac{p}{p-1}} \le (2p)^{\frac{p}{p-1}}\phi_n(x),
\end{align*}
and for any $\gamma\in (0,1)$,
\begin{equation*}
 \phi_n''(x)|x|^{2\gamma}\le C_{p,\gamma} (\phi_n(x))^{1-\frac1p(1-\gamma)},
\end{equation*}
where $C_{p,\gamma}$ is a constant  depending only on $(p,\gamma)$.
\end{lemma}
%
%
%

\begin{theorem}\label{t.5.2}
In addition to the conditions {\bf (H)}, {\bf (A1)} and {\bf (B1)} we assume
\begin{description}
\item{(i)} $|g(t,y,z)|^2\le C(g^2(t,0,0)+|y|^{2\gamma})+\alpha_t(\omega)|z|^2 $\,,  $\gamma \in (0,1)$,  $\alpha_t(\omega)q(t,X_t,X_t)\le \alpha$ a.s. for some constant $\alpha<1$;
\item{(ii)}\   for some $p>1, \xi \in L^{2p}(\Omega,\mathcal F_T,\mathbb P), f(\cdot,0,0)\in M^{2p}([0,T];\R)$ and $g(\cdot, 0,0)\in Q^{2p}([0,T];\R).$
\end{description}
Then
\begin{equation}
\E\left(\sup_{0\le t\le T}|Y_t|^{2p}+\left(\int_0^T|Z_t|^2dt\right)^{ p}\right)<\infty\,.
\label{e.yz}
\end{equation}
\end{theorem}
\begin{proof}
Let $\phi_n(x)$ be the one defined in Lemma \ref{l.5.1}. Note that $\phi_n(x)$ is convex  and $d\phi_n'(x)=\phi_n''(x)dx$. Applying Lemma \ref{itolemma} to $\phi_n(Y_t)$, we have
\begin{align}\label{ito}
&\phi_n(Y_t)+\frac12\int_t^T \phi_n''(Y_s)|Z_s|^2ds
= I_1+\int_t^T \phi_n'(Y_s)\left[g(s,Y_s,Z_s)\ola B(ds, X_s)-Z_s dW_s\right]\,,
\end{align}
where
\begin{align}
I_1:=&\phi_n(\xi)+\int_t^T \phi_n'(Y_s)f(s,Y_s,Z_s)ds+\frac12\int_t^T \phi_n''(Y_s)g^2(s,Y_s,Z_s)q(s,X_s,X_s)ds\notag \\
\le & |\xi|^{2p}+C\int_t^T |\phi_n'(Y_s)|(|f(s,0,0)|+|Y_s|+|Z_s|)ds\nonumber\\
&+C\int_t^T \phi_n''(Y_s)(g^2(s,0,0)+|Y_s|^{2\gamma})q(s,X_s,X_s)ds+\frac12\alpha\int_t^T \phi_n''(Y_s)|Z_s|^2ds.\nonumber
\end{align}
Therefore,
\begin{align}\label{e.5.9}
&\phi_n(Y_t)+\frac12\int_t^T \phi_n''(Y_s)|Z_s|^2ds\notag\\
\le &|\xi|^{2p}+C\int_t^T |\phi_n'(Y_s)|(|f(s,0,0)|+|Y_s|+|Z_s|)ds\nonumber\\
&+C\int_t^T \phi_n''(Y_s)(g^2(s,0,0)+|Y_s|^{2\gamma})q(s,X_s,X_s)ds+\frac12 \alpha\int_t^T \phi_n''(Y_s)|Z_s|^2ds\nonumber\\
&+\int_t^T \phi_n'(Y_s)\left[g(s,Y_s,Z_s)\ola B(ds, X_s)-Z_s dW_s\right].
\end{align}

Since $\phi_n'(x)$ is at linear growth when $|x|$ goes to infinity, by the conditions on $g$ and the facts $(Y,Z)\in \mathcal M^{2,\beta}, Y\in S^2([0,T];\R)$, we can prove
\[\E\int_t^T \phi_n'(Y_s)\left[g(s,Y_s,Z_s)\ola B(ds, X_s)-Z_s dW_s\right]=0\]
in the same spirit of proof for equation \eref{integrable}.
Therefore, by taking expectation of \eref{e.5.9} we have
\begin{align}
&\E(\phi_n(Y_t))+\frac12\E\int_t^T \phi_n''(Y_s)|Z_s|^2ds\notag\\
\le & \E|\xi|^{2p}+C\E\int_t^T |\phi_n'(Y_s)|(|f(s,0,0)|+|Y_s|+|Z_s|)ds\nonumber\\
&+C\E\int_t^T \phi_n''(Y_s)(g^2(s,0,0)+|Y_s|^{2\gamma})q(s,X_s,X_s)ds+\frac12 \alpha\int_t^T \phi_n''(Y_s)|Z_s|^2ds\nonumber\\
\le & \E|\xi|^{2p}+C\E\int_t^T \left(|\phi_n(Y_s)|^{\frac{2p-1}{2p}}|f(s,0,0)|+\phi_n(Y_s)+\left(\phi_n(Y_s)\right)^{\frac12}\left(\phi_n''(Y_s)\right)^{\frac12}|Z_s|)\right) ds\nonumber\\
&+C\E\int_t^T \left((\phi_n(Y_s))^{\frac{p-1}{p}}g^2(s,0,0)+(\phi_n(Y_s))^{1-\frac1p(1-\gamma)}\right)q(s,X_s,X_s)ds+\frac12 \alpha\int_t^T \phi_n''(Y_s)|Z_s|^2ds,\label{e.5.10}
\end{align}
where  the last inequality follows from Lemma \ref{l.5.1}. By Young's inequality, we have
\begin{equation}\label{e.5.11}
|\phi_n(Y_s)|^{\frac{2p-1}{2p}}|f(s,0,0)|\le \frac{2p-1}{2p} |\phi_n(Y_s)|+\frac{1}{2p}|f(s,0,0)|^{2p},
\end{equation}
\begin{equation}\label{e.5.12}
     \left(\phi_n(Y_s)\right)^{\frac12}\left(\phi_n''(Y_s)\right)^{\frac12}|Z_s|\le \frac1a\phi_n(Y_s)+a\phi_n''(Y_s)|Z_s|^2,\quad  \forall a>0,
    \end{equation}
and
    \begin{eqnarray}
&&(\phi_n(Y_s))^{\frac{p-1}{p}} g^2(s,0,0)q(s,X_s,X_s)+(\phi_n(Y_s))^{1-\frac1p(1-\gamma)}q(s,X_s,X_s)
\label{e.5.13} \\
&\le& \frac{p-1}{p}\phi_n(Y_s)+\frac 1p  |g(s,0,0)|^{2p}q(s,X_s,X_s)^{ p}+\frac{p-1+\gamma}{p}\phi_n(Y_s)
+\frac{1-\gamma}{p}q(s,X_s,X_s)^{\frac{p}{1-\gamma}}\,.\nonumber
\end{eqnarray}
Choosing sufficiently small $a$, we may find constants $\theta<1$ and $C<\infty$ independent of $n$, such that after substituting \eref{e.5.11}-\eref{e.5.13} into \eref{e.5.10}, by Lemma \ref{exponential} and conditions for $f(\cdot, 0,0)$ and $g(\cdot, 0,0)$, we have
\begin{align}\label{ito-1}
 &\E (\phi_n(Y_t))+\frac12\E\int_t^T \phi_n''(Y_s)|Z_s|^2ds\notag\\
\le &C+C \int_t^T \E\left(\phi_n(Y_s) \right)ds +\frac12\theta \E\int_t^T \phi_n''(Y_s) |Z_s|^2 ds.
\end{align}
All the terms in the above equation are finite, since $\phi_n(x)$ is at quadratic growth when $|x|\to \infty$ and $\phi_n''(x)$ is bounded. Then it follows from the  Gronwall's Lemma that, for all $n\in \mathbb N^+$, $$\sup_{0\le t\le T}\E \left(\phi_n(Y_t)\right)<Ce^{CT},$$
and hence
$$\sup_{0\le t\le T}\E \left(\phi_n(Y_t)\right)+\E\int_0^T \phi_n''(Y_s)|Z_s|^2ds<C,$$
for some constant $C$ independent of $n$.
Letting $n$ go to infinity and noting that the derivatives $\phi_n^{(i)}(x)\nearrow (|x|^{2p})^{(i)}$ for $i=0,1,2$, we have
\begin{equation}\label{e5.14}
\sup_{0\le t\le T}\E(|Y_t|^{2p})+\E\int_0^T |Y_t|^{2p-2}|Z_t|^2dt<\infty.
\end{equation}

Therefore, by \eref{ito} and the Burkholder-Davis-Gundy  inequality, we have
\begin{align*}
\E\left(\sup_{0\le t\le T}\phi_n(Y_t)\right)\le & C+ C\E\left(\int_0^T (\phi_n'(Y_s))^2\left(g^2(s,Y_s,Z_s)q(s,X_s,X_s)+|Z_s|^2\right)ds\right)^\frac12\\
\le &C + C\E\left[\sup_{0\le t\le T}(\phi_n(Y_t))^{\frac12}\left(\int_0^T |Y_s|^{2p-2}\left(g^2(s,Y_s,Z_s)q(s,X_s,X_s)+|Z_s|^2\right)ds\right)^\frac12\right]\\
\le & C+\frac12\E\left(\sup_{0\le t\le T}\phi_n(Y_t)\right)+ C\E\int_0^T |Y_s|^{2p-2}\left(g^2(s,Y_s,Z_s)q(s,X_s,X_s)+|Z_s|^2\right)ds\\
\le & A+\frac12\E\left(\sup_{0\le t\le T}\phi_n(Y_t)\right),
\end{align*}
where $A<\infty$ is a constant independent of $n$, noting that
\begin{align*}
&\E\int_0^T |Y_s|^{2p-2}\left(g^2(s,Y_s,Z_s)q(s,X_s,X_s)+|Z_s|^2\right)ds\\
\le & \E\int_0^T |Y_s|^{2p-2}\left(\left[C (g(s,0,0)^2+ |Y_s|^{2\gamma} )+\alpha_t |Z_s|^2\right] q(s,X_s,X_s)+|Z_s|^2\right)ds\\
\le& C\E\int_0^T |Y_s|^{2p-2} g^2(s,0,0) q(s,X_s,X_s) ds +C\E\int_0^T |Y_s|^{2p-2+2\gamma}ds +2 \E \int_0^T |Y_s|^{2p-2}|Z_s|^2ds
\end{align*}
is finite  by H\"older's inequality, \eqref{e5.14} and the condition that $g(\cdot, 0,0)\in Q^{2p}([0,T];\R)$.

Thus we have, letting $n\to \infty,$
\begin{equation}
\E\left(\sup_{0\le t\le T}|Y_t|^{2p}\right)<\infty.\label{y-inequality}
\end{equation}

Let $(Y^{(n)}, Z^{(n)})_{n\in \mathbb N}$, where $(Y^{(0)},Z^{(0)})=(0,0)$, be a sequence of  processes generated by the mapping \eref{mapping}. Then $(Y^{(n)}, Z^{(n)})$ converges to the solution of \eqref{bdsde1} in the space $\mathcal M^{2,\beta}$ by Theorem \ref{BDSDE}. Denote $(Y^{(n-1)}, Z^{(n-1)})$ by $(y,z)$ and $(Y^{(n)}, Z^{(n)})$ by $(Y,Z)$. Applying It\^o's formula to $Y_T^2$, we have
\begin{align*}
\int_0^T |Z_s|^2ds
\le & |\xi|^2 +2\int_0^T Y_s f(s,y_s,z_s)ds+\int_0^T g^2(s,y_s,z_s)q(s,X_s,X_s)ds\\
&+2\int_0^T Y_s\left[g(s,y_s,z_s)\ola B(ds, X_s)-Z_s dW_s\right].
\end{align*}

For any $\delta>0$,  using condition (i) and the fact that we can find $C(\delta)$ for any $\delta>0$ such that $(b+a_1+\dots+a_n)^p\le (1+\delta)b^p+C(\delta) (a_1^p+\dots+a_n^p)$, we have
\begin{align*}
&\left(\int_0^T |Z_s|^2ds\right)^p\le (1+\delta) \left(\alpha\int_0^T |z_s|^2ds\right)^p+C(\delta)\Bigg[1+|\xi|^{2p}+\left|\int_0^T Y_sg(s,y_s,z_s)\ola B(ds, X_s)\right|^p\\
&   +\left(\int_0^T (g^2(s,0,0)+|y_s|^{2\gamma})q(s,X_s,X_s)ds\right)^{p}+ \left(\int_0^T |Y_s| |f(s,y_s,z_s)|ds \right)^p+ \left| \int_0^T Y_s Z_s dW_s \right|^p\Bigg].
\end{align*}

Taking expectation on both sides and noting that $\sup\limits_n\sup\limits_{0\le t\le T}\E|Y_t^{(n)}|^{2p}<\infty$, by \eqref{y-inequality}, Lemma \ref{exponential} and conditions (i), (ii), we can find some constants $C'(\delta)$ and $C''(\delta)$
 such that
\begin{align*}
&\left(\int_0^T |Z_s|^2ds\right)^p\\
\le& (1+\delta) \left(\alpha\int_0^T |z_s|^2ds\right)^p+C'(\delta)\Bigg[1+\E\left(\int_0^T |Y_s| (f(s,0,0)+|y_s|+|z_s|)ds \right)^p \\
&+ \E\left(\int_0^T |Y_s|^2(1+|y_s|^{2\gamma}+\alpha_s |z_s|^2)q(s,X_s,X_s)ds\right)^\frac p2
+\E\left( \int_0^T |Y_s|^2 |Z_s|^2ds \right)^\frac p2\Bigg]\\
\le& (1+\delta) \alpha^p\E\left(\int_0^T |z_s|^2ds\right)^p+C''(\delta)\Bigg[1+\E\left(\int_0^T |Y_s| |z_s|ds \right)^p\\
&\qquad\qquad\qquad\qquad +\E\left( \int_0^T |Y_s|^2 |z_s|^2ds \right)^\frac p2+\E\left( \int_0^T |Y_s|^2 |Z_s|^2ds \right)^\frac p2\Bigg].
\end{align*}

By Young's inequality, for any $a>0$,
\begin{align*}
&\E\left(\int_0^T |Y_s| |z_s|ds \right)^p+\E\left( \int_0^T |Y_s|^2 |z_s|^2ds \right)^\frac p2+\E\left( \int_0^T |Y_s|^2 |Z_s|^2ds \right)^\frac p2\\
\le& \frac 1a\left[\E\left(\int_0^T  |Y_s|^2ds \right)^p+2T\E \sup_{0\le s\le T }|Y_s|^{2p}\right]+2a \E\left(\int_0^T  |z_s|^2ds \right)^p+a \E\left(\int_0^T  |Z_s|^2ds \right)^p\\
\le & C(a)+2a \E\left(\int_0^T  |z_s|^2ds \right)^p+a \E\left(\int_0^T  |Z_s|^2ds \right)^p.
\end{align*}

Therefore,
\begin{align*}
(1- aC''(\delta))  \E\left(\int_0^T |Z_s|^2ds\right)^p \le [(1+\delta) \alpha^p+2aC''(\delta)]\E\left(\int_0^T |z_s|^2ds\right)^p+A(\delta,a).
\end{align*}
 One can choose $\delta$ and $a$ small enough such that $(1+\delta)\alpha^p+3aC''(\delta)<1$. Denote $\rho=\frac{(1+\delta)\alpha^p+2aC''(\delta)}{1-aC''(\delta)}$ and $A'(\delta, a)=\frac{A(\delta,a)}{1-aC''(\delta)}$, then $0<\rho<1$ and $A'(\delta,a)<\infty$, and we have
\begin{equation*}
\E\left(\int_0^T |Z_s^{(n)}|^2ds\right)^p\le \rho \E\left(\int_0^T |Z_s^{(n-1)}|^2ds\right)^p +A'(\delta,a).
\end{equation*}
This yields
\begin{equation}\label{z-inequality}
\E\left(\lim_{n\to\infty}\int_0^T |Z_s^{(n)}|^2ds\right)^p\le \liminf_{n\to\infty} \E\left(\int_0^T |Z_s^{(n)}|^2ds\right)^p<\infty.
\end{equation}
The inequality \eref{e.yz} now follows from \eref{y-inequality} and \eref{z-inequality}.\hfill
\end{proof}

\setcounter{equation}{0}
\section{BDSDEs and semilinear SPDEs }

In this section, under proper conditions, we will obtain the regularity of the solution to the BDSDE, and  then establish the relationship between the SPDE
\begin{align}
u(t,x)=\phi(x)+&\int_t^T \left[\mathscr L u(s,x)+f(s,x,u(s,x), \nabla u(t,x)\sigma(x))\right]ds\notag\\
&+\int_t^T g(s,x,u(s,x), \nabla u(t,x)\sigma(x))\ola B(ds,x),\, t\in [0, T], \label{spde'}
\end{align}
and the BDSDE
\begin{align}
Y_s^{t,x}=&\phi(X_T^{t,x})+\int_s^T f(r, X_r^{t,x},Y_r^{t,x},Z_r^{t,x})dr\notag\\
&\qquad+\int_s^T g(r, X_r^{t,x},Y_r^{t,x},Z_r^{t,x})\ola B(dr,X_r^{t,x})-\int_s^T Z_r^{t,x}dW_r, \quad s\in[t,T].\label{bdsde'}
\end{align}

Assume the following condition for $f$ and $g$.
\begin{description}
\item[{\bf (B2)} ] Let $f, g: [0,\infty)\times \R^d\times \R\times \R^d\to \R$ be two given functions satisfying for $t\in[0,T]$
\[|f(t,x,y_1, z_1)-f(t, x, y_2,z_2)|^2\le K(|y_1-y_2|^2+|z_1-z_2|^2);\]
\[|g(t,x,y_1, z_1)-g(t, x, y_2,z_2)|^2\le K|y_1-y_2|^2+\alpha_t(x)|z_1-z_2|^2,\]
where $\alpha_t(x)q(t,x,x)\le \alpha$ for some constant $\alpha\in(0,1).$
\end{description}

\begin{theorem}\label{thm1}Assume  {\bf (H)}, {\bf (A1)}, {\bf (B2)} and that $\phi$ is of class $C^2$. Let $\{u(t,x);0\le t\le T, x\in \R^d\}$ be a random field such that $u(t,x)$ is $\cF_{t,T}^B$-measurable for each $(t,x)$, $u\in C^{0,2}([0,T]\times\R^d;\R)$ a.s., and $u(t,x)$ satisfies (\ref{spde'}). Then $u(t,x)=Y_t^{t,x}$, where $(Y_s^{t,x},Z_s^{t,x})_{t\le s\le T}$ is the unique solution to (\ref{bdsde'}).
\end{theorem}

\begin{proof}  We shall borrow the idea of the proof for \cite[Theorem 3.1]{pardouxpeng}. To prove the result, it suffices to show that $(u(s, X_s^{t,x}),\sigma(x)^T\nabla u(s,X_s^{t,x}); t\le s\le T)$ solves BDSDE \eref{bdsde'}. Letting $t=t_0<t_1<t_2<\cdots<t_n=T$,  by It\^o's formula and equation \eref{spde'}, we have
\begin{align*}
&\sum_{i=0}^{n-1} [u(t_{i},X_{t_{i}}^{t,x})-u(t_{i+1},X_{t_{i+1}}^{t,x})]\\
=&\sum_{i=0}^{n-1} [u(t_{i},X_{t_{i}}^{t,x})-u(t_{i},X_{t_{i+1}}^{t,x})]+ \sum_{i=0}^{n-1} [u(t_{i},X_{t_{i+1}}^{t,x})-u(t_{i+1},X_{t_{i+1}}^{t,x})]\\
=& \sum_{i=0}^{n-1} \Bigg[-\int_{t_i}^{t_{i+1}} \mathscr L u(t_i, X_s^{t,x})ds-\int_{t_i}^{t_{i+1}} \sigma(x)^T\nabla u(t_i, X_s^{t,x}) dW_s\\
&\quad\quad+ \int_{t_i}^{t_{i+1}} \left[\mathscr L u(s,X_{t_{i+1}}^{t,x})+f(s,X_{t_{i+1}}^{t,x}, u(s,X_{t_{i+1}}^{t,x}), \sigma(x)^T\nabla u(s, X_{t_{i+1}}^{t,x}))\right]ds\\
&\quad\quad+ \int_{t_i}^{t_{i+1}} g(s,X_{t_{i+1}}^{t,x}, u(s,X_{t_{i+1}}^{t,x}), \sigma(X_{t_{i+1}}^{t,x})^T\nabla u(s, X_{t_{i+1}}^{t,x}))\ola B(ds,X_{t_{i+1}}^{t,x})\Bigg].
%
\end{align*}
Let the mesh size go to zero and the result is concluded. \hfill
\end{proof}

To get the converse of the above theorem, we need more path regularity of $(Y,Z)$, for which we impose the following conditions. 
\begin{description}
\item{\bf (A2)} $ b\in C_{b}^3(\R^d; \R^d),\, \sigma\in C_{b}^3(\R^d; \R^{d\times d}),$ i.e. $b$ and $\sigma$ are bounded functions of class $C^3$ whose partial derivatives are also bounded.
\item{\bf (B3)} Besides condition {\bf (B2)}, we also assume
\begin{itemize}
\item[(i)]  For any $s\in [0,T], f(s,\cdot,\cdot,\cdot) \text { and } g(s,\cdot,\cdot,\cdot)$  are of class  $C^3$,
and all their partial derivatives are  bounded on  $[0,T]\times\R^d\times\R\times \R^{ d}$.

\item[(ii)] $g$ is uniformly bounded, $ |g_z(t,x,y,z)|^2 q(t,x,x) \le \alpha <1,$ and   $|g_y(t,x,y,z)|^2 q(t,x,x)<C<\infty$, for $(t,x,y,z)\in [0,T]\times\R^d\times\R\times \R^{d} $.

\item[(iii)] $\phi\in C_p^3(\R^d; \R),$  i.e. $\phi$ is of class $C^3$ whose partial derivatives are of polynomial growth.
\end{itemize}
\end{description}

Under the condition {\bf (A2)}, it is known (see e.g. \cite{Stroock}) that the random field $\{X_s^{0, x};0\le s\le T, x\in \R^d\}$ has a version of class $C^2$ in $x$, and of class $C^1$ in $(s,x)$. Moreover, for fixed $(t,x)$,
\begin{equation}
\sup_{t\le s\le T}(|X_s^{t,x}|+|\nabla X_s^{t,x}|+|\nabla^2 X_s^{t,x}|)\in\bigcap_{p\ge 1}L^p(\Omega).
\end{equation}

First we establish the relationship between $Y$ and $Z$. Denote by $D=(D^1, D^2,\cdots, D^d)$ the Malliavin derivative operator with respect to  the Brownian motion $W=(W^1, W^2, \cdots, W^d)$.

\begin{proposition}\label{p.6.2}
Assume {\bf (H), (A2)} and {\bf (B3)}. Then $Z_s^{t,x}=D_s Y_s^{t,x}$ a.s., and furthermore  the random process $\{Z_s^{t,x}, t\le s\le T\}$ has a continuous version given by
\[Z_s^{t,x}=\nabla Y_s^{t,x}(\nabla X_s^{t,x})^{-1} \sigma (X_s^{t,x}),\]
and in particular,
\[Z_t^{t,x}=\nabla Y_t^{t,x}\sigma(x),\]
where $\nabla X_s^{t,x}:=\left(\frac{\partial (X_s^{t,x})_i}{\partial x_j}\right)_{1\le i,j\le d}$ is the matrix of first order derivatives of  $X_s^{t,x}$ with respect to the initial value $x$ of $X^{t,x}_s$, and
$(\nabla Y_s^{t,x}, \nabla Z_s^{t,x})$ is the unique solution to BDSDE (\ref{e.6.6'}).
Moreover, for $ p\ge 1$,
\begin{equation}\label{e.6.9}
\E\left(\sup_{0\le s\le t} |Z_s^{t,x}|^p\right)<\infty.
\end{equation}
\end{proposition}
\begin{remark} In this proposition, $(\nabla Y_s^{t,x}, \nabla Z_s^{t,x})$ is just a notation for the solution to BDSDE (\ref{e.6.6'}). After the regularity of $Y_s^{t,x}$ is obtained in  Theorem \ref{t.6.2}, one shall see that $(\nabla Y_s^{t,x}, \nabla Z_s^{t,x})$ also means the matrix of the first order derivatives of $(Y_s^{t,x}, Z_s^{t,x})$.
\end{remark}
\begin{proof}
The proof is similar to the combination of the proofs of Proposition 2.2, Lemma 2.5 and Lemma 2.6 in \cite{pp}. Here we provide a sketch of the proof for the reader's convenience.

For a general function $\phi(x)=(\phi_1(x), \cdots, \phi_d(x))^T:\R^d\to \R^d,$ we denote
\[\phi'(x):=\left(\frac{\partial \phi_i(x)}{\partial x_j}\right)_{1\le i, j\le d}.\]
By the chain rule for vector-valued functions, $\nabla X_s^{t,x}$ is the unique solution to the following linear SDE,
\[\nabla X_s^{t,x}=I +\int_t^s b'(X_r^{t,x})\nabla X_r^{t,x}dr+\sum_{k=1}^d \int_t^s \sigma_k'(X_r^{t,x})\nabla X_r^{t,x}dW_r^k, \, s\in [t,T],\]
where $\sigma_k$ is the $k$-th column of the matrix $\sigma$. On the other hand, the Malliavin derivative $D_\theta X_s^{t,x}$ satisfies the following linear SDE, for $s\in [t, T],$
\[D_\theta X_s^{t,x}=\sigma(X_\theta^{t,x})+\int_\theta^s b'(X_r^{t,x}) D_\theta X_r^{t,x}dr+\sum_{k=1}^d \int_\theta^s \sigma_k'(X_r^{t,x})D_\theta X_r^{t,x}dW_r^k, \ \ \theta \in[t,s].\]
By the uniqueness of the solutions to linear SDEs, we have
\begin{equation}\label{e.6.10}
D_\theta X_s^{t,x}=\nabla X_s^{t,x} (\nabla X_\theta^{t,x})^{-1}\sigma (X_\theta^{t,x}), \, t\le \theta < s \le T.
\end{equation}

Let $(\nabla Y_s^{t,x},\nabla Z_s^{t,x})$ be the unique solution to  the linear BDSDE, for $ s\in [t,T],$
\begin{align}
\nabla Y_s^{t,x}=&\phi'(X_T^{t,x})\nabla X_T^{t,x}+\int_s^T \bigg[ f_x(r, X_r^{t,x}, Y_r^{t,x}, Z_r^{t,x}) \nabla X_r^{t,x}\notag\\
&+ f_y(r, X_r^{t,x}, Y_r^{t,x}, Z_r^{t,x}) \nabla Y_r^{t,x}+f_z(r, X_r^{t,x}, Y_r^{t,x}, Z_r^{t,x}) \nabla Z_r^{t,x}\bigg]dr\notag\\
&+\int_s^T \bigg[ g_x(r, X_r^{t,x}, Y_r^{t,x}, Z_r^{t,x}) \nabla X_r^{t,x}+ g_y(r, X_r^{t,x}, Y_r^{t,x}, Z_r^{t,x}) \nabla Y_r^{t,x}\notag\\
&+g_z(r, X_r^{t,x}, Y_r^{t,x}, Z_r^{t,x}) \nabla Z_r^{t,x}\bigg]B(dr,X_r^{t,x})-\int_s^T \nabla Z_r^{t,x}dW_r.\label{e.6.6'}
\end{align}

 On the other hand, as in the proof of Proposition 2.2 in \cite{pp}, one can show that under the conditions {\bf (H), (A2)} and  {\bf (B3)},  $X_s^{t,x}, Y_s^{t,x}$ and  $Z_s^{t,x}$ are in $\mathbb D^{1,2}$, and
$(D_\theta Y_s^{t,x}, D_\theta Z_s^{t,x})$ solves uniquely the linear BDSDE (\ref{e.6.6'}) as well, i.e. for $t\le \theta< s\le T,$
\begin{align*}
D_\theta Y_s^{t,x}=&\phi'(X_T^{t,x})D_\theta X_T^{t,x}+\int_s^T \bigg[ f_x(r, X_r^{t,x}, Y_r^{t,x}, Z_r^{t,x}) D_\theta X_r^{t,x}\\
&+ f_y(r, X_r^{t,x}, Y_r^{t,x}, Z_r^{t,x}) D_\theta Y_r^{t,x}+f_z(r, X_r^{t,x}, Y_r^{t,x}, Z_r^{t,x}) D_\theta Z_r^{t,x}\bigg]dr\\
&+\int_s^T \bigg[ g_x(r, X_r^{t,x}, Y_r^{t,x}, Z_r^{t,x}) D_\theta X_r^{t,x}+ g_y(r, X_r^{t,x}, Y_r^{t,x}, Z_r^{t,x}) D_\theta Y_r^{t,x}\\
&+g_z(r, X_r^{t,x}, Y_r^{t,x}, Z_r^{t,x}) D_\theta Z_r^{t,x}\bigg]B(dr,X_r^{t,x})-\int_s^T D_\theta Z_r^{t,x}dW_r.
\end{align*}

By \eref{e.6.10} and  the uniqueness of the solution to  linear BDSDEs, we have
\[D_\theta Y_s^{t,x} =\nabla Y_s^{t,x} (\nabla X_\theta^{t,x})^{-1} \sigma (X_\theta^{t,x}), \,t\le\theta< s\le T.\]
Letting $s$ decrease to $\theta$, we have
\[Z_\theta^{t,x}=\lim_{s\to\theta^+}D_\theta Y_s^{t,x}=\nabla Y_\theta^{t,x} (\nabla X_\theta^{t,x})^{-1} \sigma (X_\theta^{t,x}).\]
The continuity of $Z_s^{t,x}$ follows from the continuities of $\nabla Y_s^{t,x}, \nabla X_s^{t,x}$ and  $X_s^{t,x}$. Finally, we may obtain $L^p$ estimates for $\sup_{0\le s\le t}|\nabla Y_s^{t,x}|$ as we have done for $\sup_{0\le s \le t}|Y_s^{t,x}|$ in Theorem \ref{t.5.2}, and  \eref{e.6.9} is deduced. \hfill
\end{proof}

\begin{theorem}\label{t.6.2}
Assume  {\bf (H), (A2)} and {\bf (B3)}, and additionally assume that for some $\gamma>0$ and $K>0,$
\[ |q(t,x,x)-q(t,x,y)|\le K|x-y|^\gamma,\, \text{ for all } t\in [0,T].\]
Then the random field $\{Y_s^{t,x};0\le s\le t\le T, x\in \R^d\}$ has a version whose trajectories belong to $C^{0,0,2}([0,T]^2\times \R^d).$
\end{theorem}

\begin{proof} The proof follows from the approach used in the proof of Theorem 2.1 in \cite{pardouxpeng}, from which we also borrow some notations.

First we show that for fixed $(t,x)\in[0,T]\times \R^d,\{Y_s^{t,x}; s\in [t,T]\}$ has a continuous version. For $t\le s_1\le s_2\le T$ and $p>1,$
\begin{align*}
\E(|Y_{s_2}^{t,x}-Y_{s_1}^{t,x}|^{2p})\le& C\Bigg[\E\left(\int_{s_1}^{s_2} f(r, X_r^{t,x},Y_r^{t,x},Z_r^{t,x})dr\right)^{2p}\\
&+\E \left(\int_{s_1}^{s_2} g(r, X_r^{t,x},Y_r^{t,x},Z_r^{t,x})\ola B(dr,X_r^{t,x})\right)^{2p}+\E\left(\int_{s_1}^{s_2} Z_r^{t,x}dW_r\right)^{2p}\Bigg]\\
\le & C\Bigg[\left(\int_{s_1}^{s_2} \left(\E |f(r, X_r^{t,x},Y_r^{t,x},Z_r^{t,x})|^{2p}\right)^{1/2p}dr\right)^{2p}\\
&+\left(\E \int_{s_1}^{s_2} g(r, X_r^{t,x},Y_r^{t,x},Z_r^{t,x})^2 q(r, X_r^{t,x},X_r^{t,x}) dr\right)^{p}+\left(\E\int_{s_1}^{s_2} |Z_r^{t,x}|^2 dr\right)^{p}\Bigg]\\
\le &C |s_1-s_2|^p,
\end{align*}
where the last inequality follows from the boundedness of $f$ and $g$, the integrability of $q(r, X_r^{t,x},X_r^{t,x})$ by Lemma \ref{exponential}, and the integrability of $ Z$ by Proposition \ref{p.6.2}. Then the continuity follows from the Kolmogorov's continuity theorem.

For $t_1, t_2\in [s,T],$ denoting $X_r^i=X_r^{t_i, x_i},Y_r^i=Y_r^{t_i, x_i}, Z_r^i=Z_r^{t_i, x_i}, i=1,2,$ we have
\begin{align}
Y_s^2-Y_s^1=& \int_0^1 \phi'(X_T^1+\lambda (X_0^2-X_0^1))d\lambda~ (X_T^2-X_T^1)\notag\\
&+ \int_s^T \int_0^1 f_x(r, X_r^1+\lambda(X_r^2-X_r^1), Y_r^1, Z_r^1)(X_r^2-X_r^1)d\lambda dr\notag\\
&+ \int_s^T \int_0^1 f_y(r, X_r^2, Y_r^1+\lambda(Y_r^2-Y_r^1),  Z_r^1)(Y_r^2-Y_r^1)d\lambda dr\notag\\
&+ \int_s^T \int_0^1 f_z(r, X_r^2, Y_r^2, Z_r^1+\lambda(Z_r^2-Z_r^1))(Z_r^2-Z_r^1)d\lambda dr\notag\\
&+ \int_s^T \int_0^1 g_x(r, X_r^1+\lambda(X_r^2-X_r^1), Y_r^1, Z_r^1)(X_r^2-X_r^1)d\lambda~ \ola B(dr, X_r^2)\notag\\
&+ \int_s^T \int_0^1 g_y(r, X_r^2, Y_r^1+\lambda(Y_r^2-Y_r^1),  Z_r^1)(Y_r^2-Y_r^1)d\lambda~\ola B(dr, X_r^2)\notag\\
&+ \int_s^T \int_0^1 g_z(r, X_r^2, Y_r^2, Z_r^1+\lambda(Z_r^2-Z_r^1))(Z_r^2-Z_r^1)d\lambda ~ \ola B(dr, X_r^2)\notag\\
&  +\int_s^T g(r,X_r^1, Y_r^1, Z_r^1)[\ola B(dr, X_r^2) -\ola B(dr, X_r^1)]\notag\\
&+\int_s^T Z_r^2-Z_r^1 dW_r. \label{e.6.4}
\end{align}
We adopt the following notations:

$q_{i,j}(r):=q(r, X_r^{i},X_r^{j}),\, i,j=1,2,$
\[A_r:=\text{summation of the inner integrals of the 2nd, 3rd and 4th terms on the right-hand side of \eref{e.6.4}}\]
and
\[B_r:=\text{summation of the inner integrals of the 5th, 6th and 7th terms on the right-hand side of \eref{e.6.4}}.\]
For $p>1$, applying It\^o's formula to $|Y_s^{2}-Y_s^{1}|^{2p}$ and taking expectation, we have
\begin{align}
&\E(|Y_s^2-Y_s^1|^{2p})+p(2p-1)\E\int_s^T |Y_r^2-Y_r^1|^{2p-2}|Z_r^2-Z_r^1|^2 dr\notag\\
=& \E \Bigg\{ \left| \int_0^1 \phi'(X_T^1+ \lambda (X_T^2-X_T^1))d\lambda (X_T^2-X_T^1)\right|^{2p}+2p \int_s^T |Y_r^2-Y_r^1|^{2p-1} A_r dr\notag\\
&\quad +p(2p-1)\int_s^T |Y_r^2-Y_r^1|^{2p-2}\bigg ( B_r^2 q_{2,2}(r)+ g^2(r, X_r^1, Y_r^1, Z_r^1)[q_{1,1}(r)+q_{2,2}(r)-2q_{1,2}(r)]\notag\\
&\qquad\qquad\qquad+2B_r g(r, X_r^1, Y_r^1,Z_r^1)[q_{2,2}(r)-q_{1,2}(r)]\bigg)dr\Bigg\}. \label{e.6.5}
\end{align}
By the boundedness of the derivatives of $f$ and the fact that $|Y_r^2-Y_r^1|^{2p-1}|Z_r^2-Z_r^1|\le \frac1a |Y_r^2-Y_r^1|^{2p} + a |Y_r^2-Y_r^1|^{2p-2} |Z_r^2-Z_r^1|^2 $ for $a>0$, the second term on the right-hand-side can be bounded by
\begin{align}
&C\int_s^T \left(|Y_r^2-Y_r^1|^{2p}+|Y_r^2-Y_r^1|^{2p-1}|X_r^2-X_r^1|
+|Y_r^2-Y_r^1|^{2p-1}|Z_r^2-Z_r^1|\right)dr\notag\\
\le & C(a) \int_s^T \left(|Y_r^2-Y_r^1|^{2p}+|X_r^2-X_r^1|^{2p}\right)dr
+a\int_s^T  |Y_r^2-Y_r^1|^{2p-2}|Z_r^2-Z_r^1|^2dr\label{e.6.6}
\end{align}
for some $a\in(0,1)$.

 By (ii) in  condition {\bf (B3)} and the fact that $g_x$ is bounded, we have
 $$B_r^2 q_{2,2}(r)\le C\left[q_{2,2}(r)|X_r^2-X_r^1|^2+|Y_r^2-Y_r^1|^2\right]+\alpha'|Z_r^2-Z_r^1|^2$$
 for some $\alpha'\in (\alpha,1)$. Furthermore, noting that $|q_{i_1, j_1}(r)-q_{i_2, j_2}(r)|\le C|X_r^2-X_r^1|^\gamma$ for $(i_1, j_1)\neq(i_2, j_2)$ and the fact that $g$ and its first derivatives are bounded, we may bound the third term on the right-hand side of \eref{e.6.5} by, for some $\alpha''\in (\alpha', 1)$,
\begin{align}
&p(2p-1)\int_s^T |Y_r^2-Y_r^1|^{2p-2} \bigg[C\big(1+q_{2,2}(r)\big)|X_r^2-X_r^1|^2+C|X_r^2-X_r^1|^\gamma \notag\\
&\qquad\qquad\qquad \qquad\qquad+C|X_r^2-X_r^1|^{2\gamma}+C|Y_r^2-Y_r^1|^2+\alpha''|Z_r^2-Z_r^1|^2\bigg] dr\notag\\
\le &C \int_s^T\bigg[|Y_r^2-Y_r^1|^{2p}+(1+q_{2,2}(r)\big)^p |X_r^2-X_r^1|^{2p} + C(|X_r^2-X_r^1|^{p\gamma}+|X_r^2-X_r^1|^{2p\gamma})\bigg] dr\notag\\
&\qquad \qquad  +p(2p-1)\alpha'' \int_s^T |Y_r^2- Y_r^1|^{2p-2}|Z_r^2-Z_r^1|^{2}dr. \label{e.6.7}
\end{align}
Noting that $q_{2,2}(r)$ has finite $p$-moments for any positive $p$, combining \eref{e.6.5}, \eref{e.6.6} and \eref{e.6.7} and choosing $a$ sufficiently small, we may find $\beta\in(0,1)$ and $C>0$ such that for some $p'>p$,
\begin{align*}
&\E(|Y_s^2-Y_s^1|^{2p})+p(2p-1)\beta \int_0^s |Y_r^2-Y_r^1|^{2p-2} |Z_r^2-Z_r^1|^2dr\\
\le & C\Bigg[ \left(\E\left[\sup_{0\le t\le T}|X_t^1|^q+\sup_{0\le t\le T}|X_t^2|^{q}\right]\right)^{1/2} \left(\E |X_T^2-X_T^1|^{4p}\right)^{1/2}\\
&+\int_s^T \left(\E(|X_r^2-X_r^1|^{2p'})\right)^{p/p'}+ \E(|X_r^2-X_r^1|^{p\gamma}+|X_r^2-X_r^1|^{2p\gamma})dr+\int_s^T \E(|Y_r^2-Y_r^1|^{2p})dr  \Bigg],
\end{align*}
where $q$ is determined by the polynomial growth of $\phi'(x)$ and $p$. By Gronwall's inequality and the following estimate, for $m>0$,
\[\E(\sup_{0\le s\le T}|X_s^2-X_s^1|^{m})\le C_m (1+|x_1|^{m}+|x_2|^{m})(|x_2-x_1|^{m}+|t_2-t_1|^{m/2}), \]
we deduce that for $|x_1|\vee|x_2|\le R, t_1,t_2\in [s,T]$ there exists a constant $C_{p, R, T}$ such that
\begin{equation}
\E(|Y_s^2-Y_s^1|^{2p}) \le C_{p,R,T}(|x_2-x_1|^{p\gamma}+|t_2-t_1|^{p\gamma/2}), \ \ \forall p \ge 1. \label{e.6.8}
\end{equation}
Therefore by Kolmogorov's continuity theorem, for any fixed $s\in[0,T)$, the process $\{Y_s^{t,x}, t\in [s,T], x\in \R^d\}$ has a continuous version. Actually, using a similar argument as in the  proof of Theorem \ref{t.5.2}, we may get for any $p\ge1$,
\begin{equation}\label{e.6.9'}
\E(\sup_{0\le s\le T}|Y_s^2-Y_s^1|^{2p}) +\E\left[\left(\int_0^T|Z_s^2-Z_s^1|^2ds\right)^{p}\right] \le C(|x_2-x_1|^{p\gamma}+|t_2-t_1|^{p\gamma/2}).
\end{equation}
Now we show the existence of a continuous version of the first derivative of $Y_s^{t,x}$ in $x$. Let $\{e_1,\cdots, e_d\}$ be an orthonormal basis of $\R^d$ and $h\neq 0$ be a constant. Define
\[\Delta_h^iY_s^{t,x}:= \frac1h(Y_s^{t, x+he_i}-Y_s^{t,x}),\]
and similarly define $\Delta_h^i X_s^{t,x}$ and $\Delta_h^i Z_s^{t,x}$. Setting $x_1=x, x_2=x+he_i$ and $t_1=t_2$, by \eref{e.6.9'}, we have
\begin{equation}\label{e.6.10'}
\E(\sup_{0\le s\le T} |\Delta_h^i Y_s^{t,x}|^{2p})+\E \left[\left(\int_0^T|\Delta_h^i Z_s^{t,x}|^2ds\right)^{p}\right]  <\infty.
\end{equation}

Finally, we consider $\Delta_h^iY_s^{t,x}-\Delta_{h'}^i Y_s^{t,x'}, |x|\vee|x'| \le R$, which satisfies an equation analogous to \eref{e.6.4}.
With the help of (\ref{e.6.9'}), (\ref{e.6.10'}), the condition {\bf(B3)}  and the following
two estimations
\begin{align}
&\E\left[\int_s^T f(r) B(dr, X_r^{t, x+h e_i})-\int_s^T f(r) B(dr, X_r^{t, x'+h' e_i})\right]^{2p}\notag\\
\le  &C \E\left[\int_s^T f^2(r) \left(q(r, X_r^{t, x+h e_i},X_r^{t, x+h e_i}) + q(r, X_r^{t, x'+h' e_i},X_r^{t, x'+h' e_i})-2q(r, X_r^{t, x+h e_i},X_r^{t, x'+h' e_i})\right)dr\right]^{p}\notag\\
\le &C\E\left[\int_s^T f^2(r) \left|X_r^{t, x+h e_i}-X_r^{t, x'+h' e_i}\right|^\gamma dr\right]^{p}\label{e.6.11'}
\end{align}
for  an $(\mathcal F_t)$-adapted function $f$ and
\[\E(\sup_{0\le s \le T}|\Delta_h^iX_s^{t,x}-\Delta_{h'}^i X_s^{t,x'}|^{m})\le C(1+|x|^{m}+|x'|^{m}) (|x-x'|^{m}+|h-h'|^{m}),\]
we may get the following estimate in much the same spirit as to get estimate \eref{e.6.8}:
\begin{equation*}
\E(|\Delta_h^iY_s^{t,x}-\Delta_{h'}^i Y_s^{t,x'}|^{2p})\le C(|h-h'|^{p\gamma}+|x-x'|^{p\gamma}),
\end{equation*}
where $C$ is a constant depending on $p, R$, $T$, the bounds for $g$ and the derivatives of $f$ and $g$. This implies that the derivative of $Y_s^{t,x}$ as well as a continuous version of it exists. Moreover,  similarly as in the proof of Theorem \ref{t.5.2}, we also have the following
\begin{equation*}
\E(\sup_{0\le s\le T}|\Delta_h^iY_s^{t,x}-\Delta_{h'}^i Y_s^{t,x'}|^{2p})+\E\left[\left(\int_0^T |\Delta_h^iZ_s^{t,x}-\Delta_{h'}^i Z_s^{t,x'}|^{2}ds\right)^p\right]\le C(|h-h'|^{ p\gamma}+|x-x'|^{p\gamma}),
\end{equation*}
which implies that the derivative of $Z_s^{t,x}$ with respect to $x$ exists and it is continuous in the mean-square sense. Finally, the existence of a continuous second derivative of $Y_s^{t,x}$ can be proven in a similar way. \hfill
\end{proof}

The following result provides a nonlinear Feynman-Kac formula for  SPDE \eref{spde'}.
\begin{theorem}\label{sz5} Assume the same conditions as in Theorem \ref{t.6.2}. Then $\{u(t,x):= Y_t^{t,x};0\le t\le T, x\in \R^d\}$ is the unique classical solution to SPDE \eref{spde'}.
\end{theorem}
\begin{proof}
Uniqueness follows from Theorem \ref{thm1}, and we show that $u(t,x)=Y_t^{t,x}$ is a solution to \eref{spde'}. Noting that $u(t+h, X_{t+h}^{t,x})=Y_{t+h}^{t+h, X_{t+h}^{t,x}}=Y_{t+h}^{t,x},$ applying It\^o's formula and using \eref{bdsde'}, we have that for $h>0$,
\begin{align}
u(t+h,x)-u(t,x)=&u(t+h, X_t^{t,x})-u(t+h, X_{t+h}^{t,x})+u(t+h, X_{t+h}^{t,x})-u(t, x)\notag\\
=&-\int_t^{t+h} \mathscr L u(t+h, X_s^{t,x})ds - \int_t^{t+h} \nabla u  (t+h, X_s^{t,x}) \sigma(X_s^{t,x})dW_s \notag\\
&-\int_t^{t+h} f(s, X_s^{t,x}, Y_s^{t,x}, Z_s^{t,x}) ds-\int_t^{t+h} g(s, X_s^{t,x}, Y_s^{t,x}, Z_s^{t,x})\ola B(ds, X_s^{t,x})\notag\\
&+\int_t^{t+h} Z_s^{t,x} dW_s.\label{e.6.11}
\end{align}
Let $\pi_n$ be a partition $0=t_0<t_1<\cdots<t_n=t$. By \eref{e.6.11}, we have
\begin{align*}
\phi(x)-u(t,x)=& -\sum_{i=0}^{n-1} \int_{t_i}^{t_{i+1}} [\mathscr Lu(t_i, X_s^{t,x})+f(s,X_s^{t,x},Y_s^{t,x}, Z_s^{t,x})]ds\\
&- \sum_{i=0}^{n-1} \int_{t_i}^{t_{i+1}} g(s,X_s^{t,x},Y_s^{t,x}, Z_s^{t,x})\ola B(ds,X_s^{t,x})\\
&+ \sum_{i=0}^{n-1} \int_{t_i}^{t_{i+1}} [Z_s^{t,x}-\nabla u (t_i,X_s^{t,x})\sigma(X_s^{t,x})]dW_s.
\end{align*}
If we let mesh sizes of the partitions $\pi_n$ go to zero, by Theorem \ref{t.6.2} and Proposition \ref{p.6.2}, we have
\begin{align*}
u(t,x)=&\phi(x)+\int_t^T [\mathscr L u(s,x)+f(s,x,u(s,x),\nabla u(s,x)\sigma(x))]ds\\
&+ \int_t^T g(s,x,u(s,x), \nabla u(s,x)\sigma(x))\ola B(ds,x),
\end{align*}
and the proof concludes. \hfill
\end{proof}

\begin{remark}
Consider the linear BDSDE
\[Y_s^{t,x}=\phi(X_T^{t,x})+\int_s^T (h_r+\alpha_r Y_r^{t,x})dr+\int_s^T \beta_r Y_r^{t,x} \ola B(dr, X_r^{t,x})-\int_s^T Z_r^{t,x}dW_r, t\le s\le T,\]
where $h_r=h(r,X_r^{t,x}), \alpha_r=\alpha(r,X_r^{t,x})$ and  $\beta_r=\beta(r,X_r^{t,x})$. The solution is given by
\[Y_s^{t,x}=\phi(X_T^{t,x})\Gamma_s^T +\int_s^T \Gamma_s^r h_r dr-\int_s^T \Gamma_s^rZ_r^{t,x} dW_r,\]
where
\[\Gamma_s^r =\exp\left(\int_s^r \alpha_\tau d\tau +\int_s^r \beta_\tau \ola B(d\tau, X_\tau^{t,x})-\frac12 \int_s^r \beta_\tau^2 q(\tau, X_\tau^{t,x}, X_\tau^{t,x})d\tau\right).\]
For the corresponding SPDE,  noting that  $Y_t^{t,x}$ is $\mathcal F_{t,T}^B$-measurable, we have
\[u(t,x)=Y_t^{t,x}=\E\left[\left. \phi(X_T^{t,x})\Gamma_t^T +\int_t^T \Gamma_t^r h_r dr\right|\mathcal F_{t,T}^B\right].\]
When $h_r\equiv\alpha_r\equiv 0, \beta_r\equiv 1$ and $X_r^{t,x}=x+W_r-W_t$, the Feynman-Kac formula is given by
\begin{align*}
u(t,x)=&\E_W\Bigg[\phi(x+W_T-W_t) \exp\Bigg(\int_t^T \ola B(dr, x+W_r-W_t)\\
&\qquad \qquad -\frac12 \int_t^T q(r, x+W_r-W_t, x+W_r-W_t)dr\Bigg)\Bigg],
\end{align*}
and it coincides with the Feynman-Kac formula provided in \cite[Theorem 3.1]{hns}.
%
%
\end{remark}

\setcounter{equation}{0}

\section{Random periodic solutions to semilinear SPDEs}

In this section, we will construct random periodic solutions to semilinear SPDEs via the corresponding {\it infinite horizon} BDSDEs.  For this purpose, we first consider the solvability of the  BDSDE on $[0,T]$ with $T$ increasing to infinity,
\begin{align}\label{sz1}
\begin{cases}
&Y_s^{t,x}=Y_T^{t,x}+\int_s^T f(r, X_r^{t,x},Y_r^{t,x},Z_r^{t,x})dr\\
&\qquad\ \ \ \ -\int_s^T g(r, X_r^{t,x},Y_r^{t,x},Z_r^{t,x})\overleftarrow B(dr,X_r^{t,x})-\int_s^T Z_r^{t,x}d W_r, \quad s\in[t,T], \\
&\lim\limits_{T\to\infty}e^{-K'T}Y_T^{t,x}=0,
\end{cases}
\end{align}
 for some  positive constant $K'<\infty$.  This equation is equivalent to the following {\em infinite horizon} BDSDE,
\begin{eqnarray}\label{sz2}
e^{-K's}Y_{s}^{t,x}&=&\int_{s}^{\infty}e^{-K'r}f(r,X_{r}^{t,x},Y_{r}^{t,x},Z_{r}^{t,x})dr+\int_{s}^{\infty}K'e^{-K'r}Y_{r}^{t,x}dr\\
&&-\int_{s}^{\infty}e^{-K'r}
g(r,X_{r}^{t,x},Y_{r}^{t,x},Z_{r}^{t,x})\overleftarrow B(dr,X_r^{t,x})-\int_{s}^{\infty}e^{-K'r}Z_{r}^{t,x}dW_r.\nonumber
\end{eqnarray}

To study the equation on $[0,\infty)$, we introduce the  weighted spaces, for $p\geq2$ and $q>0$,
\begin{eqnarray*}
S^{p,-q}([0,\infty);\mathbb R)
&=&\Big\{\phi: \Omega\times [0, \infty)\to \R, \text{continuous,} \phi(t) \text{ is $\mathcal F_t$-measurable,  and }\\
&&\qquad\qquad\qquad\qquad\qquad\qquad 
   \E\big[\sup_{0\le t<\infty}e^{-qt}|\phi(t)|^p\big]<\infty\Big\}\,;\\
M^{p,-q}([0,\infty);\mathbb R^{  d})
&=&\Big\{\phi: \Omega\times[0,\infty)\to \R^d, \phi(t) \text{ is $\mathcal F_t$-measurable and }\E\int_0^\infty e^{-qt}|\phi(t)|^pdt<\infty\Big\}\,.
      \end{eqnarray*}
Assume the following conditions
\begin{description}
\item{\bf (H)'} ~ The function $q(s,x,y)$ is uniformly bounded, i.e. there exists $M<\infty$ such that
 \begin{equation*}\sup_{s\in\R_+,\, (x,y)\in \R^{2d}}|q(s,x,y)|\le M, ~~a.s.
\end{equation*}
\end{description}

\begin{description}
\item{\bf (B2)'}~ Let $f, g: [0,\infty)\times \R^d\times \R\times \R^{d}\to \R$ be two functions such that $f(\cdot, 0,0,0), g(\cdot, 0,0,0)\in M^{2,-K'}([0,\infty);\R)$ and
\[|f(t,x_1,y_1, z_1)-f(t, x_2, y_2,z_2)|^2\le K(|x_1-x_2|^2+|y_1-y_2|^2+|z_1-z_2|^2);\]
\[|g(t,x_1,y_1, z_1)-g(t, x_2, y_2,z_2)|^2\le K(|x_1-x_2|^2+|y_1-y_2|^2)+\alpha_t(x)|z_1-z_2|^2,\]
where $\alpha_t(x)q(t,x,x)\le \alpha$ for some constant $\alpha\in(0,1).$
\item{\bf (M)}\ There exists a positive
constant $\mu$ such that $2\mu-K'-{K\over{1-\alpha}}-KM>0$ and
\begin{eqnarray*}
(y_1-y_2)\big(f(t,x,y_1,z)-f(t,x,y_2,z)\big)\leq-\mu|y_1-y_2|^2
\end{eqnarray*}
with $\alpha$ and $K$ taken from condition {\bf (B2)'} and $M$ from {\bf (H)'}.
\end{description}

We will need the following estimation for $X$ (see e.g. \cite{KaratzasShreve}) in the proof the Theorem \ref{qi071}.
\begin{lemma}\label{lemma6.1} Assume  {\bf (A1)}, for $p\geq1$ and $K'>0$, we have
\begin{eqnarray*}
E\left[\int_{t}^{s}e^{-K'r}|X_{r}^{t,x}|^{2p}dr\right]\leq e^{-K't}|x|^{2p}+C E\left[\int_{t}^{s}e^{-K'r}(|b(0)|^{2p}+|\sigma(0)|^{2p})dr\right]<\infty,
\end{eqnarray*}
where $C$ is a constant only depending on given parameters.
\end{lemma}

The following theorem guarantees the existence and uniqueness of the solution to the infinite horizon BDSDE under suitable conditions.
\begin{theorem}\label{qi071} Assume  {\bf (H)', (A1), (B2)'} and {\bf (M)}, then
BDSDE (\ref{sz1}) has a unique solution $(Y,Z)\in S^{2,-K'}\bigcap
M^{2,-K'}([0,\infty);\mathbb R)
\times  M^{2,-K'}([0,\infty);\mathbb R^{d})$.
\end{theorem}
\begin{proof}  First we show the uniqueness of the solution.

Let $(Y_{s}^{t,x},Z_{s}^{t,x})$
and $(\hat{Y}_{s}^{t,x},\hat{Z}_{s}^{t,x})$ be two solutions to
BDSDE (\ref{sz1}). Denote, for $s\ge t,$
\begin{eqnarray*}
&&\bar{Y}_{s}^{t,x}=\hat{Y}_{s}^{t,x}-Y_{s}^{t,x}; \ \ \bar{Z}_{s}^{t,x}=\hat{Z}_{s}^{t,x}-Z_{s}^{t,x};\\
&&\bar{f}(s,x)=f(s,X_s^{t,x},\hat{Y}_s^{t,x},\hat{Z}_s^{t,x})-f(s,X_s^{t,x},Y_s^{t,x},Z_s^{t,x});\\
&&\bar{g}(s,x)=g(s,X_s^{t,x},\hat{Y}_s^{t,x},\hat{Z}_s^{t,x})-g(s,X_s^{t,x},Y_s^{t,x},Z_s^{t,x}).
\end{eqnarray*}
Applying It$\hat {\rm o}$'s formula  to ${\rm
e}^{-K's}{{|\bar{Y}_s^{t,x}|}^2}$ on $[s,T]$, we obtain
\begin{eqnarray*}
&&e^{-K's}{|\bar{Y}_s^{t,x}|}^2-K'\int_s^Te^{-K'r}{|\bar{Y}_r^{t,x}|}^2dr+\int_s^Te^{-K'r}|\bar{Z}_r^{t,x}|^2dr\nonumber\\
&=&e^{-K'T}{|\bar{Y}_T^{t,x}|}^2+2\int_s^Te^{-K'r}\bar{Y}_r^{t,x}\big(f(r,X_r^{t,x},\hat{Y}_r^{t,x},\hat{Z}_r^{t,x})-f(r,X_r^{t,x},Y_r^{t,x},\hat{Z}_r^{t,x})\big)dr\nonumber\\
&&+2\int_s^Te^{-K'r}\bar{Y}_r^{t,x}\big(f(r,X_r^{t,x},Y_r^{t,x},\hat{Z}_r^{t,x})-f(r,X_r^{t,x},Y_r^{t,x},Z_r^{t,x})\big)dr\nonumber\\
&&+\int_s^Te^{-K'r}|\bar{g}(r,x)|^2q(r,X_r^{t,x},X_r^{t,x})dr-2\int_s^Te^{-K'r}\bar{Y}_r^{t,x}\bar{g}(r,x)d\ola B(dr, X_r^{t,x})\nonumber\\
&&-2\int_s^Te^{-K'r}\bar{Y}_r^{t,x}\bar{Z}_r^{t,x}dW_r\nonumber\\
&\leq&e^{-K'T}{|\bar{Y}_T^{t,x}|}^2-2\mu\int_s^Te^{-K'r}|\bar{Y}_r^{t,x}|^2dr+{K\over{1-\alpha-\varepsilon}}\int_s^Te^{-K'r}|\bar{Y}_r^{t,x}|^2dr\nonumber\\
&&+(1-\alpha-\varepsilon)\int_s^Te^{-K'r}|\bar{Z}_r^{t,x}|^2dr+K\int_s^Te^{-K'r}|\bar{Y}_r^{t,x}|^2q(r,X_r^{t,x},X_r^{t,x})dr\nonumber\\
&&+\int_s^Te^{-K'r}\alpha_r(\omega)|\bar{Z}_r^{t,x}|^2q(r,X_r^{t,x},X_r^{t,x})dr-2\int_s^Te^{-K'r}\bar{Y}_r^{t,x}\bar{g}(r,x)d\ola B(dr, X_r^{t,x})\nonumber\\
&&-2\int_s^Te^{-K'r}\bar{Y}_r^{t,x}\bar{Z}_r^{t,x}dW_r,
\end{eqnarray*}
where $\varepsilon>0$ is a sufficiently small number, and the last step follows from the conditions {\bf (B2)'}, {\bf (M)} and Young's inequality.

 Taking expectation (in this proof, we shall omit the standard localization procedure for the conciseness), we have
\begin{eqnarray}\label{qi33}
&&\E^{-K's}{|\bar{Y}_s^{t,x}|}^2+(2\mu-K'-{K\over{1-\alpha-\varepsilon}}-KM)\int_s^Te^{-K'r}{|\bar{Y}_r^{t,x}|}^2dr+\varepsilon\int_s^Te^{-K'r}|\bar{Z}_r^{t,x}|^2dr\nonumber\\
&\leq&\E^{-K'T}{|\bar{Y}_T^{t,x}|}^2.
\end{eqnarray}
Taking $K''>K'$ such that $2\mu-K''-{K\over{1-\alpha-\varepsilon}}-KM$ as well,
we can see that (\ref{qi33}) remains true with $K'$ replaced by $K''$.
In particular,
\begin{eqnarray*}
\E^{-K''s}{|\bar{Y}_s^{t,x}|}^2\leq
\E^{-K''T}{|\bar{Y}_T^{t,x}|}^2,
\end{eqnarray*}
and hence
\begin{eqnarray}\label{qi34}
\E^{-K''s}{|\bar{Y}_s^{t,x}|}^2\leq {\rm
e}^{-(K''-K')T}\E^{-K'T}{|\bar{Y}_T^{t,x}|}^2.
\end{eqnarray}
Since $\hat{Y}_{s}^{t,x}, Y_{s}^{t,x}\in S^{2,-K'}\bigcap
M^{2,-K'}([0,\infty);\mathbb{R})$,
\begin{eqnarray*}
\sup_{T\geq0}\E{\rm
e}^{-K'T}{|\bar{Y}_T^{t,x}|}^2\leq
\E\sup_{T\geq0}{\rm
e}^{-K'T}(2{|\hat{Y}_T^{t,x}|}^2+2{|{Y}_T^{t,x}|}^2)<\infty.
\end{eqnarray*}
Therefore, letting $T$ go to infinity in (\ref{qi34}), we have
\begin{eqnarray*}
\E{\rm
e}^{-K''s}{|\bar{Y}_s^{t,x}|}^2=0,
\end{eqnarray*}
which yields the uniqueness.

Now we deduce the existence of the solution. For each $n\in\mathbb{N}$, we define a
sequence of BDSDEs (\ref{bdsde'}) with $\phi=0$ and $T=n$ and denote it
by BDSDE ($\ref{bdsde'}_n$). It is easy to verify that for each
$n$, the BDSDE satisfies the conditions in  Theorem \ref{BDSDE}. Therefore, for each $n$,   the unique solution $({Y}_s^{t,x,n},{Z}_s^{t,x,n})$ of
BDSDE ($\ref{bdsde'}_n$) belongs to $S^{2}([t,n];\mathbb{R})\times
M^{2}([t,n];\mathbb{R}^{d})$ which
is identical with the space
 $S^{2,-K}([t,n];\mathbb{R})\times
M^{2,-K}([t,n];\mathbb{R}^{d})$.
Let $({Y}_s^{t,x,n}, {Z}_s^{t,x,n})=(0, 0)$ for $s\in (n,\infty)$, then
$({Y}_s^{t,x,n},{Z}_s^{t,x,n})\in S^{2,-K}\bigcap
M^{2,-K}([t,\infty);\mathbb{R})\times
M^{2,-K}([t,\infty);\mathbb{R}^{d})$.
In the following, we will prove that $({Y}_s^{t,x,n},{Z}_s^{t,x,n})$ is a Cauchy sequence.

Let $(Y_{s}^{t,x,m},Z_{s}^{t,x,m})$ and
$({Y}_{s}^{t,x,n},{Z}_{s}^{t,x,n})$ be the solutions to
BDSDE ($\ref{bdsde'}_m$) and BDSDE ($\ref{bdsde'}_n$), respectively, and assume $m> n$. Denote, for $s\in [t, \infty),$
\begin{eqnarray*}
&&\bar{Y}_{s}^{t,x,m,n}={Y}_{s}^{t,x,m}-{Y}_{s}^{t,x,n}, \ \
\bar{Z}_{s}^{t,x,m,n}={Z}_{s}^{t,x,m}-{Z}_{s}^{t,x,n}.
\end{eqnarray*}
We will estimate $(Y_s^{t,x, m, n}, Z_s^{t,x,m,n})$ for $s\in [n,m]$ and $s\in[t,n]$, respectively.

{\bf (I).} When $n\leq s\leq m$,
$(\bar{Y}_{s}^{t,x,m,n},\bar{Z}_{s}^{t,x,m,n})=({Y}_{s}^{t,x,m}, {Z}_{s}^{t,x,m})$. Note that
$({Y}_s^{t,x,m},{Z}_s^{t,x,m})$ is the solution to BDSDE ($\ref{bdsde'}_m$), i.e. for $s\in[t,m]$,
\begin{eqnarray*}
\left\{\begin{array}{l}
d{Y}_{s}^{t,x,m}=-f(s,X_{s}^{t,x},{Y}_{s}^{t,x,m},{Z}_{s}^{t,x,m})ds+g(s,X_{s}^{t,x},{Y}_{s}^{t,x,m},{Z}_{s}^{t,x,m})\overleftarrow B(ds,X_s^{t,x})+{Z}_{s}^{t,x,m}dW_s\\
{Y}_{m}^{t,x,m}=0.
\end{array}\right.
\end{eqnarray*}
An application of  It\^o's
formula to $e^{-K'r}|Y_{r}^{t,x,m}|^2$ on $[s,m]$ leads to
\begin{eqnarray}\label{eq-ito}
&&e^{-K's}{|Y_s^{t,x,m}|}^2-K'\int_s^me^{-K'r}{|Y_r^{t,x,m}|}^2dr+\int_s^me^{-K'r}|Z_r^{t,x,m}|^2dr\nonumber\\
&=&2\int_s^me^{-K'r}Y_r^{t,x,m}f(r,X_{r}^{t,x},{Y}_{r}^{t,x,m},{Z}_{r}^{t,x,m})dr\\
&&+\int_s^me^{-K'r}|g(r,X_{r}^{t,x},{Y}_{r}^{t,x,m},{Z}_{r}^{t,x,m})|^2q(r,X_r^{t,x},X_r^{t,x})dr\notag\\
&&-2\int_s^me^{-K'r}Y_r^{t,x,m}g(r,X_{r}^{t,x},{Y}_{r}^{t,x,m},{Z}_{r}^{t,x,m})d\ola B(dr, X_r^{t,x})-2\int_s^me^{-K'r}Y_r^{t,x,m}Z_r^{t,x,m}dW_r.\nonumber
\end{eqnarray}
Taking expectation and using the conditions {\bf (B2)'} and {\bf (M)}, we have
\begin{eqnarray*}
&& \E \left[e^{-K's}{|Y_s^{t,x,m}|}^2-K'\int_s^me^{-K'r}{|Y_r^{t,x,m}|}^2dr+\int_s^me^{-K'r}|Z_r^{t,x,m}|^2dr\right]\\
&=&\E\Bigg[ 2\int_s^me^{-K'r}Y_r^{t,x,m}\big(f(r,X_r^{t,x},{Y}_{r}^{t,x,m},{Z}_{r}^{t,x,m})-f(r,X_r^{t,x},0,{Z}_{r}^{t,x,m})\big)dr\nonumber\\
&&+2\int_s^me^{-K'r}Y_r^{t,x,m}\big(f(r,X_r^{t,x},0,{Z}_{r}^{t,x,m})-f(r,0,0,{Z}_{r}^{t,x,m})\big)dr\nonumber\\
&&+2\int_s^me^{-K'r}Y_r^{t,x,m}\big(f(r,0,0,{Z}_{r}^{t,x,m})-f(r,0,0,0)\big)dr\nonumber\\
&&+2\int_s^me^{-K'r}Y_r^{t,x,m}f(r,0,0,0)dr+\int_s^me^{-K'r}|g(r,X_{r}^{t,x},{Y}_{r}^{t,x,m},{Z}_{r}^{t,x,m})|^2q(r,X_r^{t,x},X_r^{t,x})dr\Bigg]\nonumber\\
&\leq&\E\Bigg[-\left(2\mu-2\varepsilon-{K\over{1-\alpha-\varepsilon}}-(1+\varepsilon)KM\right)\int_s^me^{-K'r}|Y_r^{t,x,m}|^2dr+C\int_s^me^{-K'r}|X_r^{t,x}|^2dr\nonumber\\
&&+\big(1-\alpha-\varepsilon+(1+\varepsilon)\alpha\big)\int_s^me^{-K'r}|Z_r^{t,x,m}|^2dr+C\int_s^me^{-K'r}|f(r,0,0,0)|^2dr\nonumber\\
&&+C\int_s^me^{-K'r}|g(r,0,0,0)|^2dr\Bigg].
\end{eqnarray*}

Therefore,
\begin{eqnarray}\label{zhang671}
&&\E\Bigg[e^{-K's}{|Y_s^{t,x,m}|}^2+\left(2\mu-2\varepsilon-{K\over{1-\alpha-\varepsilon}}-(1+\varepsilon)KM-K'\right)\int_s^me^{-K'r}{|Y_r^{t,x,m}|}^2dr\nonumber\\
&&+(1-\alpha)\varepsilon\int_s^me^{-K'r}|Z_r^{t,x,m}|^2dr\Bigg]\\
&\leq&C\int_s^me^{-K'r}\E[|X_r^{t,x}|^2]dr+C\int_s^me^{-K'r}|f(r,0,0,0)|^2dr+C\int_s^me^{-K'r}|g(r,0,0,0)|^2dr.\nonumber
\end{eqnarray}
Note that the constant $\varepsilon>0$ can be chosen to be
sufficiently small such that all the terms on the left-hand side of
(\ref{zhang671}) are positive. By Lemma \ref{lemma6.1},  we have
\begin{eqnarray}\label{zhang674}
&&\E\int_{n}^{m}e^{-K'r}|Y_{r}^{t,x,m}|^2dr]+E\left[\int_{n}^{m}e^{-K'r}|Z_{r}^{t,x,m}|^2dr\right]\nonumber\\
&\leq&C{\rm
e}^{-K'n}|x|^{2}+C\E\int_{n}^{m}{\rm
e}^{-K'r}(|b(0)|^{2}+|\sigma(0)|^{2})dr\nonumber\\
&&+C\E\int_{n}^{m}{\rm
e}^{-K'r}(|{f}(r,0,0,0)|^2+|g(r,0,0,0)|^2)dr,
\end{eqnarray}
where the right-hand side converges to zero as $n,m\to\infty.$ Applying Burkholder-Davis-Gundy inequality to
(\ref{eq-ito}) on the interval $[n,m]$ and using a similar argument which was used to obtain (\ref{zhang674}), we have
\begin{eqnarray}\label{zhang676}
&&\E\sup_{n\leq s\leq m}e^{-K's}|{Y}_{s}^{t,x,m}|^2\nonumber\\
&\leq&C{\rm
e}^{-K'n}|x|^{2}+C\E\int_{n}^{m}{\rm
e}^{-K'r}(|b(0)|^{2}+|\sigma(0)|^{2})dr\\
&&+C\E\int_{n}^{m}e^{-K'r}(|{f}(r,0,0,0)|^2+|g(r,0,0,0)|^2)dr+C\E\int_{n}^{m}{\rm
e}^{-K'r}(|Y_{r}^{t,x,m}|^2+|Z_{r}^{t,x,m}|^2)dr,\nonumber
\end{eqnarray}
where the right-hand side goes to zero as $n,m\to\infty.$

{\bf(II).} When $t\leq s\leq n$, taking the notations
\begin{eqnarray*}
&&\bar{f}^{m,n}(s,x)=f(s,X_{s}^{t,x},{Y}_{s}^{t,x,m},{Z}_{s}^{t,x,m})-f(s,X_{s}^{t,x},Y_{s}^{t,x,n},Z_{s}^{t,x,n}),\\
&&\bar{g}^{m,n}(s,x)=g(s,X_{s}^{t,x},{Y}_{s}^{t,x,m},{Z}_{s}^{t,x,m})-g(s,X_{s}^{t,x},Y_{s}^{t,x,n},Z_{s}^{t,x,n}),
\end{eqnarray*}
we have
\begin{eqnarray*}
\bar{Y}_{s}^{t,x,m,n}={Y}_{n}^{t,x,m}+\int_{s}^{n}\bar{f}^{m,n}(r,x)dr-\int_{s}^{n}\bar{g}^{m,n}(r,x)\overleftarrow B(ds,X_s^{t,x})-\int_{s}^{n}\bar{Z}_{r}^{t,x,m,n}dW_r.
\end{eqnarray*}
Apply It$\hat {\rm o}$'s formula to ${\rm
e}^{-K'r}{{|\bar{Y}_r^{t,x,m,n}|}^2}$ on $[s,n]$, and then take expectation,
\begin{eqnarray*}
&&\E\left[e^{-K's}{|\bar{Y}_s^{t,x,m,n}|}^2-K'\int_s^ne^{-K'r}{|\bar{Y}_r^{t,x,m,n}|}^2dr+\int_s^ne^{-K'r}|\bar{Z}_r^{t,x,m,n}|^2dr\right]\nonumber\\
&=&\E\Bigg[e^{-K'n}{|{Y}_{n}^{t,x,m}|}^2\nonumber\\
&&+2\int_s^ne^{-K'r}\bar{Y}_r^{t,x,m,n}\big(f(r,X_r^{t,x},{Y}_r^{t,x,m},{Z}_r^{t,x,m})-f(r,X_r^{t,x},Y_r^{t,x,n},{Z}_r^{t,x,m})\big)dr\nonumber\\
&&+2\int_s^ne^{-K'r}\bar{Y}_r^{t,x,m,n}\big(f(r,X_r^{t,x},Y_r^{t,x,n},{Z}_r^{t,x,m})-f(r,X_r^{t,x},Y_r^{t,x,n},Z_r^{t,x,n})\big)dr\nonumber\\
&&+\int_s^ne^{-K'r}|\bar{g}^{m,n}(r,x)|^2q(r,X_r^{t,x},X_r^{t,x})dr\Bigg]\notag\\
&\leq& \E\Bigg[e^{-K'n}{|{Y}_{n}^{t,x,m}|}^2-2\mu\int_s^ne^{-K'r}|\bar{Y}_r^{t,x,m,n}|^2dr+{K\over{1-\alpha-\varepsilon}}\int_s^ne^{-K'r}|\bar{Y}_r^{t,x,m,n}|^2dr\nonumber\\
&&+(1-\alpha-\varepsilon)\int_s^ne^{-K'r}|\bar{Z}_r^{t,x,m,n}|^2dr+K\int_s^ne^{-K'r}|\bar{Y}_r^{t,x,m,n}|^2q(r,X_r^{t,x},X_r^{t,x})dr\nonumber\\
&&+\int_s^ne^{-K'r}\alpha_r(\omega)|\bar{Z}_r^{t,x,m,n}|^2q(r,X_r^{t,x},X_r^{t,x})dr\Bigg].
\end{eqnarray*}
Hence
\begin{align*}\label{zhang677}
&\E\Bigg[e^{-K's}{|\bar{Y}_s^{t,x,m,n}|}^2+\left(2\mu-K'-{K\over{1-\alpha-\varepsilon}}-KM\right)\int_{s}^{n}e^{-K'r}{{|\bar{Y}_r^{t,x,m,n}|}^2}dr\nonumber\\
&\qquad\qquad\qquad\qquad\qquad\qquad\qquad\qquad\qquad  +\varepsilon\int_s^ne^{-K'r}|\bar{Z}_r^{t,x,m,n}|^2dr\Bigg]\notag\\
\leq & \E\left[e^{-K'n}{|{Y}_n^{t,x,m}|}^2\right].
\end{align*}
Taking $\varepsilon$ small enough, we have
\begin{eqnarray}\label{zhang678}
\E\int_{s}^{n}e^{-K'r}{{|\bar{Y}_r^{t,x,m,n}|}^2}dr+\E\int_{s}^{n}e^{-K'r}|\bar{Z}_r^{t,x,m,n}|^2dr\leq C \E\left[e^{-K'n}{|{Y}_n^{t,x,m}|}^2\right],
\end{eqnarray}
where the right-hand side goes to zero as $n,m
$ go to infinity by (\ref{zhang676}).
Also by the Burkholder-Davis-Gundy  inequality, we obtain
\begin{eqnarray}\label{zhang680}
\E\sup_{t\leq s\leq n}{\rm
e}^{-K's}{{|\bar{Y}_s^{t,x,m,n}|}^2}\leq
C \E\left[e^{-K'n}{|{Y}_n^{t,x,m}|}^2\right].\end{eqnarray}

Now combining \eqref{zhang674} --  \eqref{zhang680}, we have 
\begin{eqnarray*}
&&\E\sup_{s\geq t}e^{-Ks}|\bar{Y}_{s}^{t,x,m,n}|^2+\E\int_{t}^{\infty}e^{-Kr}|\bar{Y}_{r}^{t,x,m,n}|^2dr+\E\int_{t}^{\infty}{\rm
e}^{-Kr}|\bar{Z}_{r}^{t,x,m,n}|^2dr
\end{eqnarray*}
goes to zero as $n,m\to\infty.$

Denote by $({Y}_s^{t,x},{Z}_s^{t,x})$  the limit of
$({Y}_s^{t,x,n},{Z}_s^{t,x,n})$ in the space $S^{2,-K}\bigcap
M^{2,-K}([t,\infty);\mathbb{R})\times
M^{2,-K}([t,\infty);\mathbb{R}^{d})$.
 We now show that $({Y}_s^{t,x},{Z}_s^{t,x})$ is a solution to
BDSDE (\ref{sz2}). Since
$({Y}_s^{t,x,n},{Z}_s^{t,x,n})$ satisfies BDSDE ($\ref{bdsde'}_n$), it suffices to
verify that BDSDE ($\ref{bdsde'}_n$) converges to BDSDE (\ref{sz2}) in
$L^2(\Omega)$ as $n\to\infty$. We only show
the convergence of stochastic integral with respect to $\overleftarrow B$, and the convergence of the rest terms can be proven in a similar way. To see it, notice that
\begin{eqnarray*}
&&\E\left|\int_{s}^{n}e^{-K'r}g(r,X_{r}^{t,x},Y_r^{t,x,n},Z_r^{t,x,n})\overleftarrow B(dr,X_r^{t,x})-\int_{s}^{\infty}e^{-K'r}g(r,X_{r}^{t,x},Y_r^{t,x},Z_r^{t,x})\overleftarrow B(dr,X_r^{t,x})\right|^2\\
&\leq&2\E|\int_{s}^{n}e^{-K'r}\big(g(r,X_{r}^{t,x},Y_r^{t,x,n},Z_r^{t,x,n})-g(r,X_{r}^{t,x},Y_r^{t,x},Z_r^{t,x})\big)\overleftarrow B(dr,X_r^{t,x})|^2\\
&&+2\E|\int_{n}^{\infty}e^{-K'r}g(r,X_{r}^{t,x},Y_r^{t,x},Z_r^{t,x})\overleftarrow B(dr,X_r^{t,x})|^2.
\end{eqnarray*}
As $n\to \infty,$ each term on the right-hand side of the above
inequality tends to zero, since
\begin{eqnarray*}
&&\E|\int_{s}^{n}e^{-K'r}\big(g(r,X_{r}^{t,x},Y_r^{t,x,n},Z_r^{t,x,n})-g(r,X_{r}^{t,x},Y_r^{t,x},Z_r^{t,x})\big)B(dr,X_r^{t,x})|^2\\
&\leq&C\E\int_{t}^{\infty}{\rm
e}^{-K'r}(|Y_r^{t,x,n}-Y_r^{t,x}|^2+|Z_r^{t,x,n}-Z_r^{t,x}|^2)dr
\end{eqnarray*}
and
\begin{eqnarray*}
&&\E|\int_{n}^{\infty}e^{-K'r}g(r,X_{r}^{t,x},Y_r^{t,x},Z_r^{t,x})B(dr,X_r^{t,x})|^2\\
&\leq&C\E\int_{n}^{\infty}e^{-K'r}(|X_r^{t,x}|^2+|Y_r^{t,x}|^2+|Z_r^{t,x}|^2)dr+C\int_{n}^{\infty}{\rm
e}^{-K'r}|g(r,0,0,0)|^2dr.
\end{eqnarray*}

 The proof  is concluded. \hfill
\end{proof}


 From now on, we assume the noise $(B(t,x), t\ge0, x\in\R^d)$ in SPDEs \eqref{spde} and \eqref{zhang685} is a centered Gaussian random field with covariance function
\begin{equation}\label{eq6-11}
\E[B(t,x)B(s,y)]=(t\wedge s)q(x,y),
\end{equation}
where $q(x,y)$ is a positive-definite function (see e.g. \cite[Section 5]{hns}).  The condition {\bf (H)'} for $q(x,y)$ now becomes
$\sup_{(x,y)\in \R^{2d}}|q(x,y)|\le M.$ Note that  $(B(\cdot, x), x\in\R^d)$ is a family of Brownian motions (up to a multiplicative constant) with covariance $q(x,y)$, and the joint quadratic variation is given by
$$\langle B(\cdot, x), B(\cdot, y)\rangle_t=t q(x,y). $$
We now construct a measurable metric dynamical system $(\Omega, \mathscr{F}, P, (\theta_t)_{t\ge 0})$, where
${\theta}_t:\Omega\to\Omega$ is a measurable and measure-preserving mapping defined by
${\theta}_{t}\circ {B}(s,x)={B}(s+t,x)-{B}(t,x) $ for $x\in\mathbb{R}^d$, and  ${\theta}_{t}\circ W_s=W_{s+t}-W_t$. Then for any $s,t\geq0$,
\begin{description}
\item[$(\textrm{i})$]$P(\theta^{-1}(A))=P(A), $ for all $A\in \mathcal F$;
\item[$(\textrm{i}\textrm{i})$]${\theta}_{0}=I$, where $I$ is the identity transformation on $\Omega$;
\item[$(\textrm{i}\textrm{i}\textrm{i})$]${\theta}_{s}\circ{\theta}_{t}={\theta}_{s+t}$.
\end{description}

Set, for any $\mathscr{F}$-measurable mapping $\phi$ defined on $\Omega$,
\begin{eqnarray*}
{\theta}\circ\phi(\omega)=\phi\big({\theta}(\omega)\big).
\end{eqnarray*}

For any $r\geq0$, $s\geq t$, $x\in\mathbb{R}^d$, apply the transformation $\theta_r$ to
SDE (\ref{sde}), and then it follows that
\begin{eqnarray*}
{\theta}_r\circ
X_{s}^{t,x}=x+\int_{t+r}^{s+r}b({\theta}_r\circ
X_{u-r}^{t,x})du+\int_{t+r}^{s+r}\sigma({\theta}_r\circ
X_{u-r}^{t,x})dW_u.
\end{eqnarray*}
So by the uniqueness of the solution and a perfection procedure
(see e.g. \cite{ar}), we have
\begin{eqnarray}\label{qi18}
{\theta}_r\circ X_{s}^{t,x}=X_{s+r}^{t+r,x}\ \ {\rm
for}\ {\rm all}\ r,s,t,x,\ \ \ {\rm a.s.}
\end{eqnarray}

For a given period $\tau>0$, we consider the random periodic solution to BDSDE (\ref{sz1}). For this, we assume the following random periodic condition on the coefficients.
\begin{description}
\item{\bf (P)}\
For any $t\in[0,\infty)$, $(x,y,z)\in\mathbb{R}^d\times\R\times \R^{1\times d}$, $$f(t,x,y,z)=f(t+\tau,x,y,z)\ \ \text{and}\ \ g(t,x,y,z)=g(t+\tau,x,y,z).$$
\end{description}


\begin{prop}\label{qi031}
Assume  {\bf (H)', (A1), (B2)', (M)} and {\bf (P)}, then
the unique solution $(Y_s^{t,x}, Z_s^{t,x})_{s\geq t}$ to BDSDE (\ref{sz1}) is a ``crude''
random periodic solution, i.e.  for any $0\le t\le s$,
\begin{eqnarray*}
{\theta}_\tau\circ Y_s^{t,x}=Y_{s+\tau}^{t+\tau,x}, \ \ {\theta}_\tau\circ
Z_s^{t,x}=Z_{s+\tau}^{t+\tau,x}\ \  {\rm a.s.}
\end{eqnarray*}
In particular, for any $t\geq0$,
\begin{eqnarray}\label{qi19}
{\theta}_\tau\circ Y_{t}^{t,\cdot}=Y_{t+\tau}^{t+\tau,\cdot}\ \ \ {\rm a.s.}
\end{eqnarray}
\end{prop}
\begin{proof}
Let $\hat{B}(s,x)={B}(T'-s,x)-{B}(T',x)$ for arbitrary
$T'>0$ and $-\infty<s\leq T'$. Then $\hat{B}(s,x)$ is a local martingales with the same joint quadratic variation as ${B}(s,x)$
and $\hat{B}(0,x)=0$. Then for an $\mathscr F_t$-adapted square integrable process $\{h(s)\}_{s\geq0}$ and any $r\ge0$,
 \begin{eqnarray}\label{bmrelation}
\int_{t+r}^{T+r}h(s-r)d\overleftarrow B(s,X_s^{t+r,x})=-\int_{T'-T-r}^{T'-t-r}
h({T'-s-r})d\hat{B}(s,X_{T'-s}^{t+r,x})\ \ \ \
\end{eqnarray}
for $r\ge0.$
Applying ${\theta}_\tau$ to
$\hat{B}(s,X_u^{t,x})$,  we have
\begin{eqnarray}\label{sz3}
{\theta}_\tau\circ\hat{B}(s,x)&=&{\theta}_\tau\circ({B}({T'-s},x)-{B}({T'},x))={B}({T'-s+\tau},x)-{B}({T'+\tau},x)\nonumber\\
&=&\big({B}({T'-s+\tau},x)-{B}({T'},x)\big)-\big({B}({T'+\tau},x)-{B}(T',x)\big)\nonumber\\
&=&\hat{B}({s-\tau},x)-\hat{B}({-\tau},x).
\end{eqnarray}
So for $0\leq t\leq T\leq T'$, by (\ref{qi18}), (\ref{bmrelation}) and (\ref{sz3})
\begin{eqnarray}\label{qi5}
{\theta}_\tau\circ\int_{t}^{T}h(s)d\overleftarrow B(s,X_s^{t,x})&=&-{\theta}_\tau\circ\int_{T'-T}^{T'-t}h(T'-s)d\hat{B}(s,X_{T'-s}^{t,x})\nonumber\\
&=&-\int_{T'-T}^{T'-t}{\theta}_\tau\circ h(T'-s)d\hat{B}({s-\tau},{\theta}_\tau\circ X_{T'-s}^{t,x})\nonumber\\
&=&-\int_{T'-T}^{T'-t}{\theta}_\tau\circ h(T'-s)d\hat{B}({s-\tau},X_{T'-s+\tau}^{t+\tau,x})\nonumber\\
&=&-\int_{T'-T-\tau}^{T'-t-\tau}{\theta}_\tau\circ h(T'-s-\tau)d\hat{B}({s},X_{T'-s}^{t+\tau,x})\nonumber\\
&=&\int_{t+\tau}^{T+\tau}{\theta}_\tau\circ h(s-\tau)d\overleftarrow B(s,X_{s}^{t+\tau,x}).\notag
\end{eqnarray}
Therefore, by condition {\bf (P)},
\begin{eqnarray}\label{qi6}
&&{\theta}_\tau\circ\int_t^T g(u,X_u^{t,x},Y_u^{t,x},Z_u^{t,x})\overleftarrow B(du,X_u^{t,x})\nonumber\\
&=&\int_{t+\tau}^{T+\tau}g(u,X_u^{t+\tau,x},{\theta}_\tau\circ
Y_{u-\tau}^{t,x},{\theta}_\tau\circ Z_{u-\tau}^{t,x})\overleftarrow B(du,X_u^{t+\tau,x}).
\end{eqnarray}

We apply ${\theta}_\tau$ to BDSDE (6.1), and then get, by (\ref{qi6}),
\begin{equation}\label{qi7}
\begin{cases}
&{\theta}_\tau\circ Y_s^{t,x}={\theta}_\tau\circ Y_T^{t,x}+\int_{s+\tau}^{T+\tau}f(u,X_u^{t+\tau,x},{\theta}_\tau\circ
Y_{u-\tau}^{t,x},{\theta}_\tau\circ Z_{u-\tau}^{t,x})du\\
&\qquad\ \ \ \ -\int_{s+\tau}^{T+\tau}g(u,X_u^{t+\tau,x},{\theta}_\tau\circ
Y_{u-\tau}^{t,x},{\theta}_\tau\circ Z_{u-\tau}^{t,x})\overleftarrow B(du,X_u^{t+\tau,x})- \int_{s+\tau}^{T+\tau}{\theta}_\tau\circ Z_{u-\tau}^{t,x}d W_u\\
&\lim\limits_{T\to\infty}e^{-K'(T+\tau)}{\theta}_\tau\circ Y_T^{t,x}=0.
\end{cases}
\end{equation}
On the other hand, BDSDE (\ref{sz1}) implies \begin{equation}\label{qi8}
\begin{cases}
&Y_{s+\tau}^{t+\tau,x}=Y_{T+\tau}^{t,x}+\int_{s+\tau}^{T+\tau} f(u,X_u^{t+\tau,x},Y_u^{t+\tau,x},Z_u^{t+\tau,x})du\\
&\qquad\ \ \ \ -\int_{s+\tau}^{T+\tau} g(u,X_u^{t+\tau,x},Y_u^{t+\tau,x},Z_u^{t+\tau,x})\overleftarrow B(du,X_u^{t+\tau,x})-\int_{s+\tau}^{T+\tau} Z_u^{t+\tau,x}d W_u, \quad s\in[t,T],\\
&\lim\limits_{T\to\infty}e^{-K'(T+\tau)}Y_{T+\tau}^{t+\tau,x}=0.
\end{cases}
\end{equation}
By the uniqueness of the solution to BDSDE (\ref{sz1}), it
follows from comparing (\ref{qi7}) with (\ref{qi8}) that for any $t\geq0$,
\begin{eqnarray*}
{\theta}_\tau\circ Y^{t,x}_s=Y^{t+\tau,x}_{s+\tau}, \ \
{\theta}_\tau\circ Z^{t,x}_s=Z^{t+\tau,x}_{s+\tau},\ \ {\rm for} \ s\geq t.
\end{eqnarray*}
The proof is concluded. \hfill
\end{proof}

For the infinite horizon BDSDE, we shall assume the following condition instead of {\bf (B3)}.
\begin{description}
\item{\bf (B3)'} (i) For any $t\in [0,\infty), f(t,\cdot,\cdot,\cdot) \text { and } g(t,\cdot,\cdot,\cdot)$  are of class  $C^3$,
and all their derivatives are  bounded on  $[0,\infty)\times\R^d\times\R\times \R^{d}$.

(ii) $g$ is uniformly bounded, $ |g_z(t,x,y,z)|^2 q(t,x,x) \le \alpha <1,$ and   $|g_y(t,x,y,z)|^2 q(t,x,x)<C<\infty$, for $(t,x,y,z)\in [0,\infty)\times\R^d\times\R\times \R^{d} $. 
%
\end{description}

Similar to Theorem \ref{t.6.2}, the solution to the infinite horizon BDSDE also possesses path regularity.
\begin{theorem}\label{sz4}
Assume  {\bf (H)', (A2)}, {\bf (B3)'} and {\bf (M)}, and additionally assume that, for some constant $K>0$ and $\gamma>0$,
\[ |q(x,x)-q(x,y)|\le K|x-y|^\gamma,\, \text{ for all }  t\in [0,T].\]
 Then the random field $\{Y_s^{t,x}; s\geq t\geq0, x\in \R^d\}$, which is the solution to BDSDE (\ref{sz1}), has a version whose trajectories belong to $C^{0,0,2}([0,\infty)^2\times \R^d).$
\end{theorem}

For $t\geq0$, define
$u(t,x)=Y_t^{t,x}$, where
$(Y_{s}^{t,x},Z_{x}^{t,x})_{s\ge t}$ is the solution to
BDSDE (\ref{sz1}), then it follows from Theorem \ref{sz4} that $u\in C^{0,2}([0,\infty)\times{\mathbb{R}^{d}})$. For arbitrary $T>0$, we consider SPDE \eref{zhang685} on the interval $[0,T]$. Note that there is no given terminal condition for SPDE \eref{zhang685}, so {\bf (B3)(iii)} is not needed for the solvability and regularity of its solution. By Theorem \ref{sz5}, $u(t,x)=Y_t^{t,x}$ is a classical solution
to SPDE \eref{zhang685} and we have the following theorem.
\begin{theorem}\label{qi044} Assume the same conditions in Theorem \ref{sz4}. Let $u(t,x)\triangleq Y_{t}^{t,x}$, where
$(Y_{s}^{t,x},Z_{s}^{t,x})_{s\ge t}$ is the solution to
BDSDE (\ref{sz1}). Then for arbitrary $T$ and $t\in[0,T]$,
$u(t,x)$ is a solution to SPDE (\ref{zhang685}).
\end{theorem}


The following theorem is the main result in this section.
\begin{theorem}\label{qi046} Assume  condition {\bf (P)} and the same conditions in Theorem \ref{sz4}. For any $T>0$, define $u(t,x)\triangleq Y_{t}^{t,x}$, where
$(Y_{s}^{t,x},Z_{s}^{t,x})_{s\ge t}$ is the solution to
BDSDE (\ref{sz1}). Then $u(t,x)$ has a version
which is a ``perfect'' random periodic solution to SPDE (\ref{zhang685}).
\end{theorem}

\begin{proof}
By Theorem \ref{qi044}, we know that $u(t,x)\triangleq
Y_{t}^{t,x}$ is the solution to
SPDE (\ref{zhang685}), so we get from (\ref{qi19}) that for any
$t\geq0$,
\begin{eqnarray*}
{\theta}_\tau\circ u({t,\cdot})=u({t+\tau,\cdot})\ \ \ {\rm a.s.}
\end{eqnarray*}
 The above is the so-called ``crude'' random periodic property for $u({t,\cdot})$.
By the continuity of $u(t,\cdot)$ in $t$, one can find an indistinguishable version of
$u({t,\cdot})$, still denoted by $u({t,\cdot})$, such that it is a ``perfect'' random periodic solution in the sense of (\ref{sz6}). \hfill
\end{proof}

Finally, we consider stationary solutions to SPDEs, in a special case that the coefficient functions $f, g: \R^d\times \R\times \R^{d}\to \R$ in {\bf (B3)'} are independent of time variable. Let $(\tilde {B}(t,x), t\in\R, x\in \RR^d )$ be a centered Gaussian field independent of $W$ with covariance $\E[\tilde B(t,x)\tilde B(s,y)]=(|t|\wedge |s|)q(x,y)$. Consider the following infinite horizon SPDE,
\begin{eqnarray}\label{ssz1}
v(t,x)&=&v(0,x)+\int_{0}^{t}[\mathscr{L}v(s,x)+f\big(x,v(s,x),(\sigma^T\nabla v)(s,x)\big)]ds\nonumber\\
&&+\int_{0}^{t}g\big(x,v(s,x),(\sigma^T\nabla v)(s,x)\big)\tilde{B}(ds, x).
\end{eqnarray}
For any $T>0$, if we choose $B(s,x)=\tilde{B}(T-s,x)-\tilde{B}(T,x)$ as the driven noise in  SPDE (\ref{zhang685}), then SPDE (\ref{zhang685}) is a time reversal of SPDE (\ref{ssz1}).
\begin{theorem}\label{qi046} Assume the same conditions in Theorem \ref{sz4}. For any  $T>0$, let $v(t,x)\triangleq Y_{T-t}^{T-t,x}$, where
$(Y_{s}^{t,x},Z_{s}^{t,x})_{s\ge t}$ is the solution to
BDSDE (\ref{sz1}) with $B(s,x)=\tilde{B}(T-s,x)-\tilde{B}(T,x)$ for $s\geq0$. Then $v(t,x)$ has a version
which is a ``perfect'' stationary solution to SPDE (\ref{ssz1}).
\end{theorem}
\begin{proof}
Notice that SPDE (\ref{zhang685}) is a time reversal transformation of SPDE (\ref{ssz1}). By Theorem \ref{sz4} $v(t,x)=Y_{T-t}^{T-t,x}$ is a solution to (\ref{ssz1}). Furthermore, one can show that
$v(t,x)$ does not depend on the choice of $T$ as  in \cite[Theorem 2.12]{zhangzhao}.

We define $\tilde{\theta}_{t}=({\theta}_{t})^{-1}$, $t\geq0$. Then $(\Omega,\mathscr{F},P)$ and $\tilde{\theta}$ constitute a new measurable metric dynamical system. Moreover, $\tilde{\theta}_{t}\circ \tilde {B}(s,x)=\tilde {B}(s+t,x)-\tilde {B}(t,x)$ since $\tilde {B}$ is the time-reversal form of ${B}$.

Since we are assuming that $f,g$ in BDSDEs (\ref{sz1}) and (\ref{zhang685}) are independent of time variable, {\bf (P)} holds for any $\tau\geq0$. Therefore, (\ref{sz6}) gives a stationary solution to SPDE (\ref{zhang685}) by perfection procedure (see e.g. \cite{ar,ar-sc}), i.e.
\begin{eqnarray*}
{\theta}_{r}\circ u({t,\cdot})=u({t+r,\cdot})\ \ \ {\rm for}\
{\rm all}\ t,r\geq0\ \ {\rm a.s.}
\end{eqnarray*}
Therefore,
\begin{eqnarray}\label{sz7}
\tilde{\theta}_r\circ v(t,x)&=&{\theta}_{-r}\circ u(T-t,x)={\theta}_{-r}\circ{\theta}_{r}\circ u(T-t-r,x)\nonumber\\&=&u(T-t-r,x)=v(t+r,x)\ \ \ \ \text{for all $t,r\geq0$ a.s.,}
\end{eqnarray}
where $T$ is chosen sufficiently large  such that $t+r\leq T$.
This yields that
$v(t,x)=Y_{T-t}^{T-t,x}$
is a ``perfect'' stationary solution to SPDE (\ref{ssz1}) with respect to the shift operator $\tilde{\theta}$.\hfill
\end{proof}

\section*{Acknowledgements}

J. Song was supported by ECS grant (project code 27302216) of the Hong Kong Research Grants Council.  Q. Zhang was supported by NSF of China (No. 11471079, 11631004), and the Science and Technology Commission of Shanghai Municipality (No. 14XD1400400).

\end{document}